\documentclass[12pt]{article}

\usepackage{verbatim} % for comments
\usepackage{thmtools}
\usepackage{thm-restate}
\usepackage[utf8]{inputenc}
\usepackage{multicol, multirow}
\usepackage{enumerate,lineno,setspace,float}
\usepackage{fullpage}
\usepackage{amsfonts, amsmath, amssymb, amsthm}
\usepackage{graphicx}
\usepackage{tikz, pgffor, pgfplots, mathtools}
\usepackage{longtable}
\usepackage{caption} % this is here just to fix longtable captions
\usetikzlibrary{arrows.meta}
\usepackage{hyperref}
\usepackage[numbers, sort&compress]{natbib}
\usepackage[capitalize]{cleveref}
\hypersetup{
    colorlinks=true,       % false: boxed links; true: colored links
    linkcolor=blue,        % color of internal links (change box color with linkbordercolor)
    citecolor=blue,        % color of links to bibliography
}

\usepackage{color}
\definecolor{VeryLightBlue}{rgb}{0.9,0.9,1}
\definecolor{LightBlue}{rgb}{0.8,0.8,1}
\definecolor{MidBlue}{rgb}{0.5,0.5,1}
\definecolor{DarkBlue}{rgb}{0,0,0.6}
\definecolor{Blue}{rgb}{0,0,1}
\definecolor{Gold}{rgb}{1,0.843,0}
\definecolor{Orange}{rgb}{1,0.5,0}
\definecolor{LightGreen}{rgb}{0.88,1,0.88}
\definecolor{MidGreen}{rgb}{0.6,1,0.6}
\definecolor{DarkGreen}{rgb}{0,0.6,0}
\definecolor{VeryLightYellow}{rgb}{1,1,0.9}
\definecolor{LightYellow}{rgb}{1,1,0.6}
\definecolor{MidYellow}{rgb}{1,1,0.5}
\definecolor{DarkYellow}{rgb}{1,1,0.01}
\definecolor{DarkPurple}{rgb}{.6,0,1}
\definecolor{Red}{rgb}{1,0,0}
\definecolor{VeryLightRed}{rgb}{1,0.9,0.9}
\definecolor{LightRed}{rgb}{1,0.8,0.8}
\definecolor{MidRed}{rgb}{1,0.55,0.55}

\newcommand{\darkBlue}[1]{{\color{MidRed}{#1}}}
\newcommand{\angel}[1]{\darkBlue{\bf [Angel: #1]}}

\allowdisplaybreaks

\newtheorem{theorem}{Theorem}
\newtheorem{lemma}[theorem]{Lemma}
\newtheorem{proposition}[theorem]{Proposition}
\newtheorem{corollary}[theorem]{Corollary}
\newtheorem{definition}{Definition} 

\newtheorem{conjecture}{Conjecture}
\newtheorem{question}{Question}

\newcommand{\Href}[1]{\hyperref[#1]{\Cref{#1}}}
\renewcommand{\href}[1]{\hyperref[#1]{\cref{#1}}}
\renewcommand{\eqref}[1]{\hyperref[#1]{(\ref{#1})}}

\def\leq{\leqslant}
\def\geq{\geqslant}

\newcommand{\p}{\mathbb{P}}

\title{Antimagic Labeling for Unions of Graphs with \\ 
Many Three-Paths\thanks{Partially supported by the National Science Foundation grant DMS 1600778.}}
\author
{
Angel Chavez\thanks{California State University Los Angeles. Current address: University of Minnesota Twin Cities. Email:  chave389@umn.edu.} \and 
Parker Le 
\thanks{California State University Los Angeles.   Email: ple31@calstatela.edu.}
\and
Derek Lin\thanks{California State University Los Angeles.   Email: dlin4@calstatela.edu}
\and 
Daphne Der-Fen Liu
\thanks{Corresponding author. California State University Los Angeles. Email:  dliu@calstatela.edu.}
\and
Mason Shurman
\thanks{Current Address: University of California, Irvine. Email:  mshurman@uci.edu.}
}

\begin{document}

\maketitle

\begin{abstract} 
Let $G$ be  a graph with $m$ edges and let $f$ be a bijection from $E(G)$ to $\{1,2, \dots, m\}$. For any vertex $v$, denote by $\phi_f(v)$ the sum of $f(e)$ over all edges $e$ incident to $v$. If $\phi_f(v) \neq \phi_f(u)$ holds for any two distinct vertices $u$ and $v$, then  $f$ is called an {\it antimagic labeling}  of $G$. We call $G$  {\it antimagic} if such a labeling exists. Hartsfield and Ringel \cite{Ringel} conjectured that all connected graphs except $P_2$ are antimagic. Denote the disjoint union of graphs $G$ and $H$ by $G \cup H$, and the disjoint union of $t$ copies of $G$ by $tG$. 
For an antimagic graph $G$ (connected or disconnected), we define  the  parameter $\tau(G)$ to be the maximum integer such that $G \cup  tP_3$ is antimagic for all $t \leq \tau(G)$. Chang, Chen, Li, and Pan showed that for all antimagic graphs $G$, $\tau(G)$ is finite \cite{shifted}.  Further, Shang, Lin, Liaw \cite{star forest} and  Li \cite{double star} found the exact value of $\tau(G)$ for special families of graphs: star forests and balanced double stars,  respectively. They did this by finding explicit antimagic labelings of $G\cup tP_3$ and proving a tight upper bound on $\tau(G)$ for these special families. In the present paper, we generalize their results by proving an upper bound on $\tau(G)$ for all graphs. For star forests and balanced double stars, this general bound is equivalent to the bounds given in \cite{star forest} and \cite{double star} and tight.  In addition, we prove that the general bound is also tight for every other graph we have studied, including an infinite family of jellyfish graphs, cycles $C_n$ with %\liu{replace "where" with "with"} 
$3 \leq n \leq 9$, and the double triangle $2C_3$.
\end{abstract}

%%%%%%%%%%%%%%%%%%%%%%%%%%%%%%%%%%

\section{Introduction}

%%%%%%%%%%%%%%%%%%%%%%%%%%%
The graphs considered in this article are not necessarily connected, unless otherwise indicated.  Let $G$ be a graph with $m$ edges. For a bijection  $f: E(G) \to \{1,2, \ldots, m\}$ and for any vertex $v$, denote by $\phi_f(v)$ the sum of $f(e)$ over all edges $e$ incident to $v$.   We call $f$ an {\it antimagic labeling} of $G$ if for any  pair of vertices $u$ and $v$, $\phi_f(u) \neq \phi_f(v)$. A graph is {\it antimagic} if it admits an antimagic labeling.  
When $f$ is clear in context, we shorten $\phi_f(v)$ to $\phi(v)$ and call it the {\it $\phi$-value} of $v$.

Antimagic labeling was introduced by Hartsfield and Ringel  \cite{Ringel}, in which the following conjecture was posed:

\begin{conjecture}
\label{conj}
{\rm \cite{Ringel}} 
Every connected graph except $P_2$ is antimagic. 
\end{conjecture}

\noindent
\Href{conj} has received much attention in the past years (cf. \cite{partition, dense, caterpillars}), and many families of graphs are known to be antimagic. Alon, Kaplan, Lev, Roditty, and Yuster \cite{dense} proved that dense graphs are antimagic. Precisely, the authors showed that graphs of order $n$ with minimum degree $\delta(G) \geq c \log n$ for some  constant $c$ or with maximum degree $\Delta(G) \geq n-2$ are antimagic. Other families of graphs known to be antimagic include  regular graphs \cite{regular2, regular, regular bipartite}, trees with at most one vertex of degree two and their subdivisions \cite{partition, trees}, caterpillars \cite{caterpillar1,  caterpillars, caterpillar all},  
spiders \cite{Shang}, and double spiders \cite{strongly}. 

While \Href{conj} has been studied extensively, antimagic labelings for disconnected graphs have received less attention. 
It is known that there exist nontrivial 
disconnected non-antimagic graphs.
For instance, it is  easy to see that the union of two copies of $P_3$ 
is not antimagic (\Href{FigureSwapping}).

\begin{figure}[h]
\centering
\begin{tikzpicture}[scale = 1.3, rotate=90]
    \newcommand{\s}{0.08}
    \newcommand{\ls}{0.15} % size of circles for phi-values
    % this ox needs to be lowered a little so it can look better
    \newcommand{\ox}{0.3} % off set x variable to make circles 
    
    \begin{scope}
    % this is the right-most 2P_3
    % path on top
	\draw(2,-1) -- (2,1);
	\draw[fill=black] (2, 1) circle (\s);
	\draw (2 - \ox, 1) circle (\ls) node {3};
	\draw[fill=black] (2, 0) circle (\s);
	\draw (2 - \ox, 0) circle (\ls) node {5};
	\draw (2 - \ox, 0) circle (1.5*\ls);
	\draw[fill=black] (2,-1) circle (\s);
	\draw (2 - \ox, -1) circle (\ls) node {2};
	\node[above] at (2, 0.5) {3};
    \node[above] at (2,-0.5) {2};
    
    % path on the bottom
    \draw(1,-1) -- (1,1);
    \draw[fill=black] (1, 1) circle (\s);
    \draw (1 - \ox, 1) circle (\ls) node {4};
    \draw[fill=black] (1, 0) circle (\s);
	\draw (1 - \ox, 0) circle (\ls) node {5};
	\draw (1 - \ox, 0) circle (1.5*\ls);
    \draw[fill=black] (1,-1) circle (\s);
	\draw (1 - \ox, -1) circle (\ls) node {1};
    \node[above] at (1, 0.5) {4};
	\node[above] at (1,-0.5) {1};
    \end{scope}
    
    \begin{scope}[shift={(0,4)}]
    % this is the middle 2P_3
    % path on top
	\draw(2,-1) -- (2,1);
	\draw[fill=black] (2, 1) circle (\s);
	\draw (2 - \ox, 1) circle (\ls) node {4};
	\draw (2 - \ox, 1) circle (1.5*\ls); 
	\draw[fill=black] (2, 0) circle (\s);
	\draw (2 - \ox, 0) circle (\ls) node {6};
	\draw[fill=black] (2,-1) circle (\s);
	\draw (2 - \ox, -1) circle (\ls) node {2};
	\node[above] at (2, 0.5) {4};
    \node[above] at (2,-0.5) {2};
    
    % path on the bottom
    \draw(1,-1) -- (1,1);
    \draw[fill=black] (1, 1) circle (\s);
    \draw (1 - \ox, 1) circle (\ls) node {3};
    \draw[fill=black] (1, 0) circle (\s);
	\draw (1 - \ox, 0) circle (\ls) node {4};
    \draw (1 - \ox, 0) circle (1.5*\ls);
    \draw[fill=black] (1,-1) circle (\s);
	\draw (1 - \ox, -1) circle (\ls) node {1};
    \node[above] at (1, 0.5) {3};
	\node[above] at (1,-0.5) {1};
    \end{scope}
    
    \begin{scope}[shift={(0,8)}]
    % this is the left-most 2P_3
    % path on top
	\draw(2,-1) -- (2,1);
	\draw[fill=black] (2, 1) circle (\s);
	\draw (2 - \ox, 1) circle (\ls) node {3};
	\draw (2 - \ox, 1) circle (1.5*\ls);
	\draw[fill=black] (2, 0) circle (\s);
	\draw (2 - \ox, 0) circle (\ls) node {7};
	\draw[fill=black] (2,-1) circle (\s);
	\draw (2 - \ox, -1) circle (\ls) node {4};
	\node[above] at (2, 0.5) {3};
    \node[above] at (2,-0.5) {4};
    
    % path on the bottom
    \draw(1,-1) -- (1,1);
    \draw[fill=black] (1, 1) circle (\s);
    \draw (1 - \ox, 1) circle (\ls) node {2};
    \draw[fill=black] (1, 0) circle (\s);
	\draw (1 - \ox, 0) circle (\ls) node {3};
    \draw (1 - \ox, 0) circle (1.5*\ls);
    \draw[fill=black] (1,-1) circle (\s);
	\draw (1 - \ox, -1) circle (\ls) node {1};
    \node[above] at (1, 0.5) {2};
	\node[above] at (1,-0.5) {1};
    \end{scope}
\end{tikzpicture}
\caption {The graph $2P_3$ is not antimagic. The graphs above exhaust all possible labelings, and prove that none are antimagic. Circled numbers are $\phi$-values and uncircled  numbers are edge  labels. Twice circled $\phi$-values are identical. 
}
\label{FigureSwapping}
\end{figure}
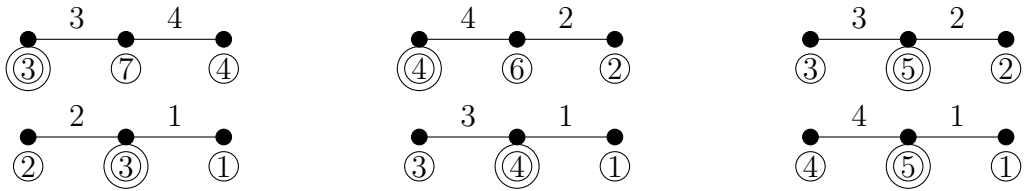
%\mason{added this paragraph about trivial cases because it lets us not have to keep saying "except for graphs containing $P_2$ as a component". This means that the generality of our statements is more clear}
%\mason{I think we have to also consider disconnected graphs containing $P_3$ as a component as a trivial case. Otherwise, take a graph $G$ for which the second bound is much smaller than the first bound. Then $\beta(G \cup P_3) \neq \tau(G\cup P_3)$}
Further, it is clear that a graph containing $P_2$ as a component is not antimagic: both vertices adjacent to the isolated edge will have the same $\phi$-value. Similarly, a graph containing two isolated vertices as components is not antimagic: both vertices will have a $\phi$-value of $0$. A graph containing exactly one isolated vertex as a component will be antimagic if and only if the graph induced by deleting that vertex is antimagic. Therefore, this paper will focus on graphs that do not contain isolated edges ($P_2$) or isolated vertices as components.% \angel{this sentence should be combined with 1st first sentence of the paper.} 

Throughout the paper, we denote the union of disjoint graphs $G$ and $H$ by $G \cup  H$, and the union of $t$ copies of $G$ by $tG$. Chang, Chen, Li, and Pan (Theorem 3.5, \cite{shifted}) showed that for any graph $G$, $G \cup t P_3$ is not antimagic for a sufficiently large $t$. It is natural to consider the following parameter for a graph $G$: 

\begin{definition}
If $G$ is antimagic, define $\tau (G)$ as the maximum non-negative integer $t$ such that $G \cup t' P_3$ is antimagic for all $0 \leq t' \leq t$. If $G$ is not antimagic,  define $\tau (G) = - \infty$.   
\end{definition}

The aim of this article is to investigate  
\begin{question} 
\label{question}
What is $\tau(G)$ for a graph $G$?     
\end{question}

The main  result of this paper is the  following general upper bound of $\tau(G)$. 
An edge is {\it internal} if both of its ends are non-leaf vertices; otherwise it is a {\it pendant} edge. 
\begin{theorem}
\label{general bound}
Let $G$ be a graph with $n$ vertices and $m$ edges including $k$ internal edges. Assume $G$ does not contain isolated vertices nor $P_2$ as a component. Let $t'$ be the number of components of $G$ isomorphic to $P_3$. If $G \cup tP_3$ is antimagic, then 
$$
t \leq \min \left\{ (3 + 2\sqrt{2})(m -n)+(1 + \sqrt{2})(m +\frac{1}{2}), \ 2m + 5(k-t') +1 \right\}.
$$
\end{theorem} 

\noindent 
Denote the floor of the right-side of the inequality in \Href{general bound} by $\beta(G)$: 
$$
\beta(G) := \left\lfloor\min \left\{ (3 + 2\sqrt{2})(m -n)+(1 + \sqrt{2})(m +\frac{1}{2}), \ 2m + 5(k-t') +1 \right\} \right\rfloor.
$$

\begin{corollary}
\label{coro main} 
%\liu{modified the statement}
Let $G$ be a graph with $n$ vertices and $m$ edges including $k$ internal edges. Assume $G$ does not contain isolated vertices nor $P_2$ as a component. Let $t'$ be the number of components of $G$ isomorphic to $P_3$. Then $\tau(G) \leq \beta(G)$. Moreover, if $\tau(G) =  \beta(G)$, then the converse also holds. That is, $G \cup tP_3$ is antimagic if and only if $t \leq \tau(G)$. 
\end{corollary}

%\noindent
The proof of  \Href{general bound} is presented in \Href{mainsection}. As mentioned above, a result in  \cite{shifted} 
implies that $\tau(G) \leq 8m+1$, where $m=|E(G)|$.  %\liu{Add a comma before "where"} %\mason{say "$m=|E(G)|$"}
\Href{general bound} implies $\tau(G) \leq 7m+1$.  

The bound of \Href{general bound} is sharp for many graphs. 
An {\it $n$-star} $S_n$, $n \geq 3$, is a tree with a center vertex $v$ and $n$ pendant edges incident to $v$. 
A {\it star forest} is a forest whose components are stars. 
All edges in a  star forest are pendant  edges, so $k=0$ in \Href{general bound}. A {\it double star} $S_{a,b}$ is a tree created by adding an edge between the centers of stars $S_a$ and $S_b$. That is, $S_{a,b}$ has $a+b+1$ edges where only one is internal  ($k=1$).  By \Href{general bound}, we obtain the following:

\begin{corollary}
\label{trees bound 3}
Let $F$ be a non-trivial forest with $m$ edges, $q$ components, $t'$ components isomorphic to $P_3$, and $k$ internal edges, where each component has at least two  edges. Then 
$$
\tau(F) \leq \min\left\{ -q(3 + 2\sqrt{2})   +(1+\sqrt{2})(m + \frac{1}{2}),\ 2m+5(k-t')+1\right\}. 
$$
\end{corollary}
\begin{corollary}
\label{star forest}
If $G$ is a star forest with $m$ edges, $q$ components and each component has at least three edges, then 
$$
\tau(G) \leq \min\{-q(3 + 2\sqrt{2})+(1 + \sqrt{2})(m  +\frac{1}{2}), \ 2m+1\}.
$$
\end{corollary}

\begin{corollary} 
%{\rm \cite{double star}} 
\label{double star} 
The double star $S_{a,b}$, $a, b \geq 1$, has 
$$
\tau(S_{a,b}) \leq \min\{-(3 + 2\sqrt{2})+(1 + \sqrt{2})(a+b+\frac{3}{2}), \ 2(a+b+4)\}.
$$
\end{corollary}

\noindent
Shang, Lin, and Liaw \cite{star forest} and Li \cite{double star} proved %\derek{the} 
bounds for  $\tau(G)$ when $G$ is a star forest and when $G$ is a double star,  respectively. Their bounds for  those graphs  coincide with \Href{star forest} and \Href{double star} respectively,  but with different expressions. 

By providing desired valid antimagic labelings, the authors of \cite{star forest} and \cite{double star} also showed the reverse direction of this inequality for star forests and balanced double stars, giving us: 
\begin{theorem}
{\rm \cite{star forest}} 
A star forest G where each component has at least three edges has $\beta(G)=\tau(G)$. 
\end{theorem}

\begin{theorem}
{\rm \cite{double star}} 
For any $a \geq 2$, $\beta(S_{a,a})=\tau(S_{a,a})$. %is tight in $\tau$. 
\end{theorem}

 In \cite{star forest, double star}, the authors studied  %\derek{were studies of} 
antimagic labelings of $G \cup tP_3$ for special trees and forests which contain no cycles. In \Href{jellysection} and \Href{2regsection}, we explore families of  graphs that do contain cycles: jellyfish graphs and 2-regular graphs. 

For positive integers  $r$ and $k$ with $k \geq 3$, the {\it jellyfish} graph  $J(C_{k}, r)$ is obtained by taking a cycle $C_k$ and adding $r$ pendant 
edges (with leaves) to each vertex on $C_k$. By \Href{general bound}, we have:

\begin{corollary}
\label{th k3 up} 
	If $J(C_k, r) \cup{t}P_3$ is antimagic, then  
	$$
	t \leq\min\left\{\left(1+\sqrt2\right)\left(kr+k+\frac{1}{2}\right), \ 2kr+7k+1\right\}.
	$$
	Consequently, if $r\geq11$, then $\tau(J(C_3, r))\leq6r+22$. 
\end{corollary}

\noindent
We prove in \Href{jellysection} that the bound of \Href{general bound} is sharp for infinitely many jellyfish graphs: 
\begin{theorem}
\label{jellyfish}
For $r\geq 11$, $J(C_3, r)\cup{t}P_3$ is antimagic if and only if $t \leq 6r+22$. Equivalently, $\beta(J(C_3,r))=\tau(J(C_3,r))$, provided $r \geq 11$.
\end{theorem}

In \Href{jellysection} and the works discussed above,   %\cite{star forest, double star, sea urchin}, 
it seems simpler to find antimagic labelings for $G \cup tP_3$ for all $t \leq \tau(G)$  when the second bound in \Href{general bound} is smaller. This occurs when $G$ has many leaves and few internal edges. 
On the other hand, it is more difficult to find such labelings when the first bound is smaller: so far to our knowledge these labelings have only been found for finitely many graphs. 

Because $\tau(G)$ has not been well studied in graphs where the first bound in \Href{general bound} is smaller, in \Href{2regsection}, we investigate 2-regular graphs, where because $m=n=k$, the first bound is always smaller.  
We start by proving some general  recursive properties for antimagic labelings of $G \cup tP_3$ for any graph $G$. These properties prove especially useful when $G$ is 2-regular.
With these results, we show that the bound in  \Href{general bound} is sharp for $2C_3$ and $C_n$, \ 
$3 \leq n \leq 9$, and we conjecture this holds for all $C_n$.  
%, and we conjecture this holds  for all $C_n$. 
%\begin{theorem}
%\label{cycles} 
%For $3 \leq n \leq 9$,  $\tau(C_n) = \beta(C_n)$.  
%Equivalently, for $3 \leq n \leq 9$, $C_n \cup tP_3$ is antimagic if and only if 
%$$
%0 \leq t \leq \left\lfloor (1 +  \sqrt{2})(n + \frac{1 }{2})    \right\rfloor.  
%$$
%\end{theorem}
%\noindent
%Further, we prove the bound is sharp for $2C_3$. 
In \Href{problems}, we ask whether the bound on $\tau(G)$ in \Href{general bound} is tight and raise other questions for future research.

%%%%%%%%%%%%%%%%%%%%%%%%%%%%%%%%%%%%%%%%%%%%
%
\section{Proof of \Href{general bound}} \label{mainsection}
%
%%%%%%%%%%%%%%%%%%%%%%%%%%%%%%%%%%%%%%%%%%%

Before we introduce \Href{main lemma}, the main lemma in the proof of \Href{general bound}, we establish necessary notations. Let $G$ be a graph with $n$ vertices and $m$ edges. Let $t$ be a non-negative integer. Suppose $f$ is an antimagic labeling for $G'= G \cup t P_3$. 
Then $G'$ has $m'=m+2t$ edges. Denote the centers (degree-2 vertices)  of the 3-paths by $\{w_1, w_2, \ldots, w_t\}$. 

Denote the {\it sum} of all the  labels we can use, $[1, m'=m+2t]$, by: 
\begin{equation}
s(G,t) := \sum\limits_{e \in E(G)} f(e) + \sum\limits_{i = 1}^{t}\phi_f(w_i)  =\sum\limits_{i=1}^{m+2t} i
= 2t^2 + (2m+1)t + \frac{m(m+1)}{2}.
\label{def s}
\end{equation} 
%\liu{added $V^*(G')$ and modified the paragraph below.} 

Denote $V^*(G')= V(G) \cup \{w_1, w_2, \cdots, w_t\}$. 
For every $u \in V^*(G')$, $\phi_f(u) \leq m+2t$ if and only if $\phi_f(u)=f(e)$ for some $e \in E(G)$. Because  
$|V^*(G')|=n+t$ and  $|E(G)|=m$, it must be that at least $t+n-m$ vertices $u \in V^*(G')$ have $\phi_f(u) \geq m+2t+1$. Denote the {\it least}  total vertex sums for these vertices by:  
\begin{equation}
    l(G,t) := \sum\limits_{i=1}^{t+n-m} (m+2t+i) = \frac{(t + n - m)(5t + n + m + 1)}{2}.
\label{def l}
\end{equation}
When $G$ and $t$ are  clear in the context, we simply denote $s(G,t)$ and $l(G,t)$ by $s$ and $l$, respectively. 

\begin{lemma}
\label{main lemma}
Let $G$ be a graph with $n$  vertices and $m$ edges. 
If $G \cup tP_3$ is antimagic for some non-negative integer $t$, then $s \geq l$.
\end{lemma}

\begin{proof}
Let $f$ be an antimagic labeling of $G' = G \cup tP_3$. Then $G'$ has $m' = m + 2t$ edges. 
For each vertex $v \in V(G)$, either $\phi(v) \geq m'+1$ or $\phi(v)=f(e)$ for some $e \in E(G)$.  When $t \geq 1$, denote the degree-2 vertices of the $t$ 3-paths by $w_i$, $1 \leq i \leq t$, so that $\phi(w_1) < \phi(w_2) < \phi(w_3) < \ldots < \phi(w_t)$.  Define the following: 
$$
\begin{array}{lll}
B &:=& \{v \in V(G) \ | \ \phi(v) \geq m'+1\}, \\ 
B'&:=& \{v \in V^*(G')\ | \ \phi(v) \geq m'+1\}, \\ 
S &:=& \{v \in V(G) \ | \ \phi(v) \leq m'\}, \\
W &:=& \{e \in E(G) \ | \ f(e) = \phi(w_i) \mbox{ for some $w_i$}\}, \\
E(S) &:=& \{e \in E(G) \ | \ f(e)= \phi(v) \ \mbox{for some $v \in S\}$}, \\
R &:=& E(G) \setminus (W \cup E(S)).
\end{array}
$$
%\liu{Added $V^*(G')$ in $B'$} 

If $t=0$, then $W = \emptyset$ and $B=B'$.  Observe that $W \cup E(S) \cup R$ is a partition of $E(G)$. Hence $m = |W| + |E(S)| + |R|$. Note that $n = |B| + |S|$ and $|S|=|E(S)|$. % and $\sum_{v \in S} \phi(v) = \sum_{e \in E(S)} f(e)$.
Therefore, we obtain  $|B| = n - m + |W| + |R|$ and 
$$
\begin{array}{llll}
\sum\limits_{v \in B} \phi(v) + \sum\limits_{v \in S} \phi(v) &=& 2 \sum\limits_{ e \in E(G)} f(e) \\ 
%&=& 
%\sum\limits_{ e \in E(G)} f(e) + \sum\limits_{ e \in E(G)} f(e) \\
&=& \sum\limits_{ e \in E(G)} f(e) + \sum\limits_{ e \in E(S)} f(e)  +\sum\limits_{e \in W} f(e) +  \sum\limits_{ e \in R} f(e) \\
&=& \sum\limits_{ e \in E(G)} f(e) + 
\sum\limits_{v \in S} \phi(v) +   
\sum\limits_{e \in W} f(e) +  \sum\limits_{ e \in R} f(e).  
\end{array} 
$$ 
Simplifying the above, we obtain 
$$
\sum\limits_{e \in E(G)} f(e) = \sum\limits_{v \in B} \phi(v) - \sum\limits_{ e \in W} f(e) - \sum\limits_{ e \in R} f(e). $$

Note that $|B'| = |B| + t - |W| = n-m+|W|+|R| + t - |W| = n-m+t+|R|$.  Substituting the above into \href{def s} and by \href{def l} we get: 

\begin{equation}
\begin{array}{llll}  
s &=& \sum\limits_{v \in B} \phi(v) + \sum\limits_{i = |W| + 1}^{t} \phi(w_i)  - \sum\limits_{e \in R} f(e) \\    
  &=& \sum\limits_{v \in B'} \phi(v) - \sum\limits_{e \in R} f(e) \\
 &\geq& \sum\limits_{i=1}^{n-m+t+|R|} (m'+i)  - \sum\limits_{e \in R} f(e) \\ 
 &\geq& \sum\limits_{i=1}^{t+n-m} (m'+i) + |R|(m'+1) - |R|m' \\
 &\geq& l.
\end{array}
\label{sl block}
\end{equation} 
Hence, the proof is complete. 
\end{proof}

\noindent
{\it Proof of the First Bound in \Href{general bound})} \ By \Href{main lemma}, if $G\cup tP_3$ is antimagic we have $s(G,t) \geq l(G,t)$, which is a quadratic inequality in $t$. Solving this inequality, we find that: 

%\mason{Because of the big O term, we can write exact equality right?}\liu{so you mean we change the last to "="?}
\begin{align*}
    t &\leq \Big(4m - 3n + \frac{1}{2}\Big) + \sqrt{(3\sqrt{2}m - 2\sqrt{2}n + \frac{1}{\sqrt{2}})^2 - \frac{1}{4}} \\ 
      &= \Big(4m - 3n + \frac{1}{2}\Big) + \Big(3\sqrt{2}m - 2\sqrt{2}n + \frac{1}{\sqrt{2}}\Big) + O\Big(\frac{1}{n+m}\Big).
\end{align*}

\begin{proposition}
\label{angel only makes good approximations}
For any integers $m$ and $n$, we have 
\[ 
\left\lfloor\sqrt{(3\sqrt{2}m - 2\sqrt{2}n + \frac{1}{\sqrt{2}})^2 - \frac{1}{4}}\right\rfloor = \left\lfloor3\sqrt{2}m - 2\sqrt{2}n + \frac{1}{\sqrt{2}}\right\rfloor.
\]
\end{proposition}
\begin{proof}
Suppose to the contrary there exists an integer $x$ such that 
$$
\sqrt{(3\sqrt{2}m - 2\sqrt{2}n + \frac{1}{\sqrt{2}})^2 - \frac{1}{4}} < x \leq 3\sqrt{2}m - 2\sqrt{2}n + \frac{1}{\sqrt{2}}.$$
Square both sides and simplify to yield $\frac{1}{4} < x^2 - (18m^2 + 8n^2 - 24mn - 4n + 6m) \leq \frac{1}{2}$, which is impossible since $m$ and $n$ are integers. 
This completes the proof of \Href{angel only makes good approximations}.
\end{proof}

\noindent
By \Href{angel only makes good approximations}, $t \leq (3 + 2\sqrt{2})(m -n) + (1 +  \sqrt{2})(m +\frac12)$, completing the proof of the first  bound in \Href{general bound}. 
\medskip

\noindent
{\it Proof of the Second Bound in \Href{general bound})} \ Next we prove the second bound, $\tau(G) \leq 2m + 5(k-t')+1$. Suppose $G' = G \cup tP_3$ is antimagic. If $t+t' \leq k$, then  $t \leq k-t' \leq 2m+5(k-t')+1$.  
Assume $t+t' > k$.  %Then $m' = |E(G')| = m + 2t$.
Let $f$ be an antimagic labeling for $G'$. %\mason{replace "by the assumption" with "since"} 
Since  $G$ contains $k$ internal edges, at least $(t+t'-k)$ 3-paths must have $\phi(w_i) \geq m'+1 =m+2t+1$. Denote by $y$  the sum of $\phi(w_i)$ for these $(t+t'-k)$ paths. Then we can bound $y$ below by using the fact that every $\phi(w_i)$ is distinct, and we can bound $y$ above by using the fact that all the edges are given different labels:
$$
\sum\limits_{i=1}^{t+t'-k} (m' + i) \leq y\leq \sum\limits_{i=1}^{2(t+t'-k)} (m'+1 -i). 
$$
As $m' = m+2t$, using direct calculation and solving the above inequalities, we obtain $t \leq 2m + 5(k-t') +1$. This completes the proof of \Href{general bound}.  
\hfill$\blacksquare$

Our next result shows that $\beta(G) \geq 0$ if every component of $G$ has at least 3 edges.
\begin{proposition}
\label{beta}
Let $G$ be a graph without isolated vertices and having no $P_2$ as a component.  If $\beta(G) < 0$, then $G$ contains at least one $P_3$ as a component and $G$ is not antimagic. \end{proposition}

\begin{proof}
We first prove that $\beta(G) \geq 0$ if $G$ does not contain $P_3$ as a component. Assume $G$ does not have $P_3$ as a component. Then every component of $G$ contains at least $3$ edges.  If $G$ has $q$ components, then $m-n \geq -q$, and $m \geq 3q$. Therefore the first bound of $\beta(G)$ is non-negative. As the second bound of $\beta(G)$ is always positive, $\beta(G) \geq 0$.  If  $G$ is antimagic, then $\beta(G) \geq \tau(G) \geq 0$. Hence, the second conclusion holds. 
\end{proof}

In general, %\parker{add a comma after "in general"}  
there is no simple way to determine which of the two bounds in \Href{general bound} is better (smaller),  
but we can make some estimates.  
Suppose $m \to \infty$.
%, and the constants are not exactly precise. \mason{the constants are exactly precise. The numbers we give aren't} 
Without loss of generality, suppose the graph does not contain $P_3$ as a component. Because $k \leq m$, the second bound is smaller  if $\frac{n}{m} < 3(5\sqrt{2}-7)\approx 0.21$. 
Furthermore, if $k \approx m$ then the second bound is better if and only if $\frac{n}{m} <  3(5\sqrt{2}-7)\approx 0.21$. This ratio is derived by substituting % \parker{I don't think we need "in" here?} 
$m$ for $k$ in the second bound and solving. Note that this if and only if statement does not hold without the assumption that $k \approx m$. For example, in star forests, where $k = 0 \not\approx m$, even though $\frac{n}{m} \approx 1 > 3(5\sqrt{2}-7)$, the second bound is smaller for sufficiently large $m$. For the cases when $n \approx m$, such as in trees, the second bound is smaller if and only if roughly $\frac{k}{m} < \frac{\sqrt{2}-1}{5} \approx 0.08$.  This ratio is derived by substituting $n$ for $m$ in the first bound. Note that in 2-regular graphs, $\frac{k}{m} = 1$, so the first bound is always smaller. 

%%%%%%%%%%%%%%%%%%%%%
%
\section{Proof of \Href{jellyfish}} \label{jellysection}
%
%%%%%%%%%%%%%

Recall a jellyfish graph $J(C_k, r)$ is established by attaching $r$ pendant edges to every vertex of a $k$-cycle $C_k$.  Throughout this section we denote the $k$ internal vertices on the jellyfish by  $v_i, i\in[1,k]$, and the $t$ internal vertices on the 3-paths in the graph $J(C_k, r) \cup tP_3$ by $w_i, i\in[1,t]$. 

To prove  \Href{jellyfish}, we start with the following lemma: 

\begin{lemma}\label{le par}
Let $k, n$ be positive integers, where $n$ and $k$ are not both even. 
Let $a_1, a_2, \ldots, a_n$ be integers, and define sets of consecutive integers by $A_i = [a_i, a_i+k-1]$, $i \in [1,n]$. Let $S$ be the multi-set-union $\cup_{i=1}^n A_i$ where repetitions are allowed (if the $A_i$'s are disjoint, then $S= \cup_{i=1}^n A_i$).  
Then  $S$ can be partitioned into multi-sets $S_1, S_2, \cdots, S_k$ so that each $S_j$ contains exactly one element from each $A_i$, $i\in[1,n]$ and the set $\{\sum S_j: j \in [1,k]\}$ consists of $k$ consecutive integers, where $\sum S_j$ is the sum of elements in $S_j$. Formally,   
    \[
    \sum S_j = \sum_{x\in S_j} x=a_1 + j-1 + \frac{1}{k}\left(\sum_{y\in S \setminus A_1}y \right).
    \]
\end{lemma}
\begin{proof}
We write $S$ as %\angel{a not an} \liu{I think it is an "an" due to the sound of $n$} 
an $n \times k$ matrix $M$, where the  $i^{th}$-row are numbers from $A_i = [a_i, a_i + k-1]$. %\liu{Added a comma before "where"} 
For each odd row $M_{2i+1}$ we write the elements from $A_{2i+1}$ in increasing order while for each even row $M_{2i}$ we write the elements from $A_{2i}$ in decreasing order. Observe that the column sums of the sub-matrix formed by any two consecutive rows $M_{i}$ and $M_{i+1}$ are identical.  Explicitly, this means for any column index $j$, $M_{ij}+M_{(i+1)j} = a_{i}+a_{i+1}+k-1$. 
Thus, if $n$ is odd, the column sums of $M$ form a set of consecutive $k$ integers, and the proof is complete by letting $S_j$ be the elements in the $j^{th}$ column, $j \in [1,k]$. 

Now assume $n$ is even. By our assumption, 
$k$ must be odd. We re-arrange the numbers in the second row $M_2$ to be: 
$$
\begin{array}{llll} 
M_2' 
%&=&a_2 + \ \left[ \ \left( \frac{k-1}{2}, \ \frac{k-1}{2}+1, \  \cdots, \ \frac{k-1}{2} + k-1 \right) \ \ \ (\mod \ k) \ \right] \\
&=&\left(a_2+ \frac{k+1}{2}, \ a_2 + \frac{k+3}{2}, \cdots, a_2+k-1, \ a_2, \ a_2+1, \ \cdots, \ a_2+ \ \frac{k-1}{2} \right).\\
\end{array} 
$$ 

%\liu{Should the second term above be $a_2$ instead of $a_1$?} 

The column sums of the sub-matrix formed by $M_1$ %\liu{Should this be $M_1$?} 
and $M_2'$ are a set of consecutive $k$ integers, $[a_1+a_2+\frac{k-1}2, \ a_1+a_2+k-1+\frac{k-1}2]$. 
%\liu{Should the two places of  $(k+1)/2$ be $(k-1)/2$?} 

By the above discussion, the column sums of the remaining $n-2$ rows (if $n \geq 4$) are identical.  Hence, the column sums of $M'$ (where $M_2$ is replaced by $M_2'$) form a set of consecutive $k$ integers.  The proof is complete by letting $S_j$ be the elements in the $j^{th}$ column of $M'$, $j \in [1,k]$.
\end{proof}

\medskip

\noindent
{\it Proof of \Href{jellyfish})} \  By \Href{th k3 up} it suffices to show that  $G'=J(C_3, r)\cup{t}P_3$ is antimagic for $t \leq 6r+22$. 
Note $|E(G')|=m'=3r+2t+3$. 

Assume $t \leq 2$. Label the internal edges on the jellyfish with $m', \ m'-1, \ m'-2$, so that the internal vertices on the jellyfish have partial sums $[2m'-3, 2m'-1]$. If $t=1$ or $t=2$, label a 3-path with $m'-3$ and $m'-4$. If $t=2$, label the remaining 3-path with $m'-5$ and $m'-6$. We then proceed to label the remaining edges of the jellyfish by applying \Href{le par} to the collection $[2m'-3, 2m'-1] \cup [1,3r] = [2m'-3, 2m'-1] \cup [1,3] \cup [4,6] \cup \dots \cup [3r-2, 3r]$ to obtain three sets with distinct sums. Assign the numbers in the set containing $2m'-3$ to the edges incident to the vertex on the cycle with a partial sum $2m'-3$. Do the same for $2m'-1$ and $2m'-2$. This will ensure the vertex sums of the three internal vertices of the jellyfish are distinct. Because $r \geq 11$,  the resulting labeling is antimagic as $\phi(v_1),\phi(v_2),\phi(v_3)>\phi(w_i) > m'$ for any  $w_i$. This completes the proof for $t \leq 2$.   

Assume $t\geq3$. In the following three steps, we (1) assign labels to the edges of the cycle $C_3$ and three $3$-paths, (2) assign labels to the pendant edges of the jellyfish, and (3) assign labels to  
the unlabeled $3$-paths.

{\bf (1) Edges on  $C_3$ and three 3-paths:} Regardless of $t$, we first fix the labels of three $3$-paths with labels $\{2,6\}$, $\{4,5\}$, and $\{3,7\}$ to obtain $\phi$-values $[8,10]$ for the internal vertices of the three $3$-paths. Next, we assign $8, 9, 10$ to  $E(C_3)$ on the jellyfish, which gives us the partial $\phi$-values of $v_1, v_2, v_3$ as $[17,19]$, as shown in \Href{th14}. 

\begin{figure}[hbtp]
    \centering
    \begin{tikzpicture}[scale = 1]
        \begin{scope}[shift={(0, 0)}, scale=2]
            \newcommand{\s}{0.05}
            
            \draw (0,0) -- (2,0) -- (1,0.866) -- (0,0);
            
            % internal nodes of jellyfish
            \draw[fill=black] (0, 0) circle (\s);
    	    \draw[fill=black] (2, 0) circle (\s);
    	    \draw[fill=black] (1, 0.866) circle (\s);
    	    
    	    % internal labels of jellyfish
            \node[left] at (0.5, 0.433) {$8$};
            \node[above] at (1, 0) {$9$};
            \node[right] at (1.5, 0.433) {$10$};
            
            % partial phi-values for internal nodes of jellyfish
            \draw (-0.3, 0.1) circle (0.15) node {17};
            \draw (1, 1.1) circle (0.15) node {18};
            \draw (2.25, 0.1) circle (0.15) node {19};
        \end{scope}
        
        \begin{scope}[shift = {(6,0)}]
            \foreach \i/\u/\d/\s in {0/2/6/8, 2/4/5/9, 4/3/7/10
            } {
                \draw (\i, 0) -- (\i, 2);
                
                \draw[fill=black] (\i, 0) circle (0.1);
                \draw[fill=black] (\i, 1) circle (0.1);
                \draw[fill=black] (\i, 2) circle (0.1);
                
                \node[left] at (\i, 0.5) {\u};
                \node[left] at (\i, 1.5) {\d};
                
                \ifnum\i=6
                    % i might want to change this into a rounded rectangle 
	                \draw (\i+0.9, 1) circle (0.6) node {\s};
	            \else
	                \draw (\i+0.6, 1) circle (0.3) node {\s};
	            \fi
            }
        \end{scope}
    \end{tikzpicture}
    \caption{Fixed labels for $J(C_3, r)\cup tP_3$. Circled numbers are (partial) $\phi$-values and other numbers are labels on edges.}
\label{th14}
\end{figure}
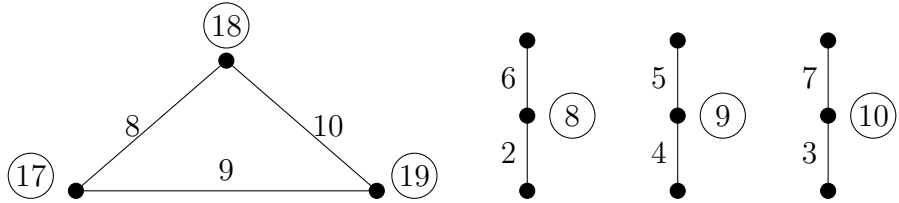
	
{\bf (2) Pendant edges on the jellyfish:} This step is split into two cases. First,  
suppose that $t$ is even 
or $t=3$. 
Label four pendant edges using labels in $\{1\} \cup [11,21]$ to each $v_i$ as shown in \Href{th13}.  
    
\begin{figure}[hbtp]
    \centering
    \begin{tikzpicture}[scale = 1]
        
        % to draw the jellyfish 
        \begin{scope}[shift={(10, 0)}, scale=2]
            \newcommand{\s}{0.05}
            
            % internal edges of jellyfish
            \draw (0,0) -- (2,0) -- (1,0.866) -- (0,0);
            
            % all pendant edges incident to (0, 0)
            \draw (-0.25,0.433) -- (0,0);
            \draw (-0.5,0) -- (0,0);
            \draw (-0.25,-0.433) -- (0,0);
            \draw (0.25,-0.433) -- (0,0);
            
            % all pendant edges incident to (2, 0)
            \draw (2.25,0.433) -- (2,0);
            \draw (2.5,0) -- (2,0);
            \draw (2.25,-0.433) -- (2,0);
            \draw (1.75,-0.433) -- (2,0);
            
            % all pendant edges incident to (1, 0.866)
            \draw (0.5,0.866) -- (1,0.866);
            \draw (0.75,1.3) -- (1,0.866);
            \draw (1.25,1.3) -- (1,0.866);
            \draw (1.5,0.866) -- (1,0.866);
            
            % internal nodes of jellyfish
            \draw[fill=black] (0, 0) circle (\s);
    	    \draw[fill=black] (2, 0) circle (\s);
    	    \draw[fill=black] (1, 0.866) circle (\s);
    	    
    	    % all leaves adjacent to (0, 0)
    	    \draw[fill=black] (-0.25,0.433) circle (\s);
    	    \draw[fill=black] (-0.5,0) circle (\s);
    	    \draw[fill=black] (-0.25,-0.433) circle (\s);
    	    \draw[fill=black] (0.25,-0.433) circle (\s);
    	    
    	    % all leaves adjacent to (2, 0)
    	    \draw[fill=black] (0.5,0.866) circle (\s);
    	    \draw[fill=black] (0.75,1.3) circle (\s);
    	    \draw[fill=black] (1.25,1.3) circle (\s);
    	    \draw[fill=black] (1.5,0.866) circle (\s);
    	    
    	    % all leaves adjacent to (1, 0.866)
    	    \draw[fill=black] (2.25,0.433) circle (\s);
    	    \draw[fill=black] (2.5,0) circle (\s);
    	    \draw[fill=black] (2.25,-0.433) circle (\s);
    	    \draw[fill=black] (1.75,-0.433) circle (\s);
    	    
    	    % internal labels of jellyfish
            \node[left] at (0.5, 0.433) {$8$};
            \node[above] at (1, 0) {$9$};
            \node[right] at (1.5, 0.433) {$10$};
            
            % labels on leaves near to (0, 0)
            \node[left] at (-0.25,0.433) {$1$};
            \node[left] at (-0.5,0) {$17$};
            \node[left] at (-0.25,-0.433) {$20$};
            \node[left] at (0.25,-0.433) {$21$};
            
            % labels on leaves near to (1, 0.866)
            \node[right] at (2.25,0.433) {$11$};
            \node[right] at (2.5,0) {$14$};
            \node[right] at (2.25,-0.433) {$15$};
            \node[right] at (1.75,-0.433) {$19$};
            
            % labels on leaves near to (2, 0)
            \node[above] at (0.5,0.866) {$12$};
            \node[above] at (0.75,1.3) {$13$};
            \node[above] at (1.25,1.3) {$16$};
            \node[above] at (1.5,0.866) {$18$};
            
            % partial phi-values for internal nodes of jellyfish
            \draw (0.3, 0.1) circle (0.15) node {76};
            \draw (1, 0.65) circle (0.15) node {77};
            \draw (1.7, 0.1) circle (0.15) node {78};
        \end{scope}
        
        \begin{scope}[shift = {(16,-1)}]
            \foreach \i/\u/\d/\s in {0/2/6/8, 2/4/5/9, 4/3/7/10} {
                \draw (\i, 0) -- (\i, 4);
                
                \draw[fill=black] (\i, 0) circle (0.1);
                \draw[fill=black] (\i, 2) circle (0.1);
                \draw[fill=black] (\i, 4) circle (0.1);
                
                \node[right] at (\i, 1) {\u};
                \node[right] at (\i, 3) {\d};
                \draw (\i+0.6, 2) circle (0.3) node {\s};
            }
        \end{scope}
    \end{tikzpicture}
    \caption{Fixed labels for $J(C_3, r)\cup tP_3$, where $t$ is even or $t=3$.} %\liu{Added a comma before "where"}}
    \label{th13}
    \end{figure}
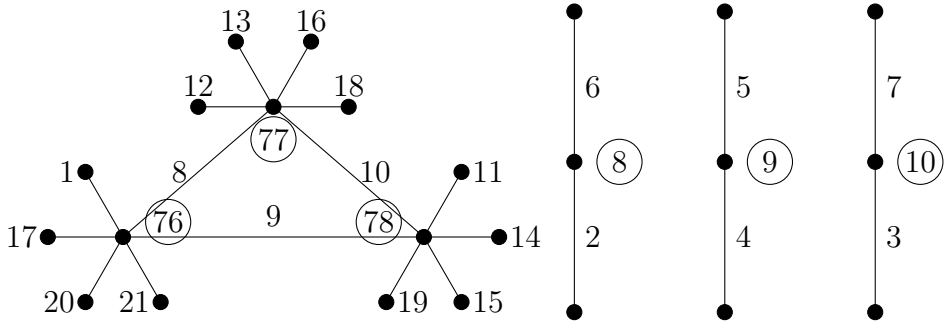
    
As the sum of the four pendant edges incident to each $v_i$ is 59, the  partial vertex sums of $v_i$ on the jellyfish are: $[76, 78]$. We then label the remaining pendant edges on the jellyfish by applying \Href{le par} with the sums $[76, 78]$ and the unused consecutive labels $[22,3r+9] = [22, 24] \cup [25, 27] \cup \dots \cup [3r+7, 3r+9]$ such that the vertex sums of the $3$ internal vertices of the jellyfish remain consecutive integers. %\angel{move since to back of sentence. 
If $t=3$, the resulting labeling is antimagic because $\phi(v_1),\phi(v_2),\phi(v_3)>m'$ as 76, 77, and 78 are in different $S_j$'s. 
%Since 76, 77, and 78 are in different $S_j$'s, if $t=3$, the resulting labeling is antimagic as $\phi(v_1),\phi(v_2),\phi(v_3)>m'$. 
If $t \geq 4$ and $t$ is even, we calculate this set of vertex sums as: 
\begin{equation}\label{vik3te}
    \left[76+\frac13\sum_{i=22}^{3r+9}i,\ 78+\frac13\sum_{i=22}^{3r+9}i\right]=\left[\frac12(3r^2+19r+28), \ \frac12(3r^2+19r+32)\right].
\end{equation}

Next, suppose $t$ is odd and $t \neq 3$. We label the pendant edges of the jellyfish by applying \Href{le par} to $[17,19] \cup [11,3r+10] = [17,19] \cup [11,13] \cup [14,16] \ldots \cup [3r+8,3r+10]$ so that the vertex sums of the three internal vertices of the jellyfish are the consecutive integers:
\begin{equation}\label{vik3to}
    \left[ 17 + \frac13\left(\sum_{i=11}^{3r+10}i\right), \ 19 + \frac13\left(\sum_{i=11}^{3r+10}i\right) \right]
    =
    \left[\frac12(3r^2+21r+34), \ \frac12(3r^2+21r+38)\right].
\end{equation}

{\bf (3) Remaining 3-paths:} In steps (1) and (2), we  have used the labels in  $[2,3r+10]$ if $t$ is odd and $t\neq3$, and labels in $[1,3r+9]$ if $t$ is even or $t=3$. 

In this step, if $t$ is odd and $t\neq3$, we first label one of the remaining 3-paths by  $\{1, m'\}$. Then for all cases, %\parker{put a comma after "cases"} 
there are an odd number of unlabeled 
%\parker{fixed typo "unlabelled" to "unlabeled"} 
$3$-paths. The unused labels are:
\[
\left\{
\begin{array}{lll}
[3r+11,3r+t+6] \cup [3r+t+7,3r+2t+2] &\mbox{if $t$ is odd and $t \neq 3$;} 
\\ 

[3r+10, 3r+t+6] \cup [3r+t+7, 3r+2t+3] &\mbox{if $t$ is even or $t=3$.} 
\end{array}
\right.
\]
Applying \Href{le par}, we partition the above unused labels into pairs 
%corresponding 
to the edges of the remaining 3-paths, so that their internal vertices have consecutive vertex sums as:

%We apply \Href{le par} by partitioning the remaining unused labels, namely, the two sets $[3r+11,3r+t+6]$ and $[3r+t+7,3r+2t+2]$ if $t$ is odd \derek{and $t\neq3$}, and $[3r+10,3r+t+6]$ and $[3r+t+7,3r+2t+3]$ if $t$ is even \derek{or $t=3$}, into pairs to the two edges in $P_3$ so that the internal vertices of the remaining $3$-paths have consecutive vertex sums:

\begin{equation}\label{wik3te}
    \left[6r+t+15+\left\lceil\frac{t}{2}\right\rceil, \ 6r+2t+11+\left\lfloor\frac{t}{2}\right\rfloor\right].
\end{equation}

%\liu{edit the following argument and format}

\noindent
To show that $f$ is an antimagic labeling, it suffices to verify: 

(i) $\phi(w_i)>m$ for $i\in[4,t]$, and 

(ii) $\phi(w_i)< \min\{\phi(v_j)\}$ for $i\in[1,t]$ and $j\in[1,3]$.

\noindent
Inequality (i) is true since by \href{wik3te} and the assumptions that $r \geq 11$ and $t \leq 6r+22$, we have:
	\[
	6r+t+15+\left\lceil \frac{t}{2} \right\rceil \geq m'+1= 3r+2t+4.
\]
Inequality (ii) is true since by the assumptions $r \geq 11$ and $t \leq 6r+22$ together with 
\href{vik3te}, \href{vik3to}, and \href{wik3te}, we obtain:  
\noindent 
	\[6r+2t+11+\left\lfloor\frac{t}{2}\right\rfloor < 
	\left\{
	\begin{array}{lll}
\frac12(3r^2+21r+34) &\mbox{if $t$ is odd;}\\
\frac12(3r^2+19r+28) &\mbox{if $t$ is even.}
\end{array}	
\right. 
	\]
This completes the proof of  \Href{jellyfish}. 
\hfill $\blacksquare$

We remark here that it has been shown in \cite{sea urchin} that the bound in \Href{th k3 up} is also tight for other jellyfish graphs $J(C_k, r)$, %\liu{added below} including the following cases: 
%%   \item $k=3$ and $r=10$, 
%\item $k \leq 7$ and $r \geq 11$, 
%\item $k \geq 8$ and $r \geq 12$. 
%\end{itemize}
%\angel{can we make this more compact by saying
including the following cases: (i) $k=3$ and $r=10$, (ii) $k \leq 7$ and $r \geq 11$, (iii) $k \geq 8$ and $r \geq 12$. 

%%%%%%%%%%%%%%%%%%%%%%%%%%%%
%
%
\section{Two-Regular Graphs} \label{2regsection}
%
%%%%%%%%%%%%%%%%%%%%%%%%%%%%

In this section, we discuss the values of $\tau(G)$ for 2-regular graphs, which  are disjoint unions of cycles of lengths %of \parker{add "of" after lengths}
at least 3. For a 2-regular graph $G$ with $m$ edges we have $m=n=k$ in \href{general bound}, implying that the first bound in \Href{general bound} is always smaller than the second. %That is,  

\begin{corollary}
\label{cycle}
Let $G$ be a 2-regular graph with $n$ vertices. Then 
$$
\tau(G) \leq \Big\lfloor (1 +  \sqrt{2})(n + \frac{1}{2}) \Big\rfloor.  \hspace{1in}
$$
\end{corollary}

\noindent
In this section, we prove that the above bound is tight for $C_n$, $3 \leq n \leq 9$, and for $2C_3$.  
To this end, we start by establishing some general  recursive properties for antimagic labelings of $G \cup tP_3$ for any graph $G$. 

\medskip
\noindent
{\bf Remark.}  Recall $s=s(G,t)$ and $l=l(G,t)$ defined in \href{def s} and \href{def l} in the proof of \href{general bound}. The calculation in  \href{sl block}  indeed shows that $s$ is at least the sum of the $t+n-m$ largest $\phi$-values of $V(G \cup tP_3)$. 
This fact can be used to prove the following results.

\begin{theorem}
\label{max phi bound}
Let $G$ be an $m$-edge $n$-vertex graph. Assume $f$ is an antimagic labeling for $G' = G \cup tP_3$ for some $t$. Then $\max\{\phi(v): v \in V(G')\} \leq n + 3t + s - l$, where $s=s(G,t)$ and $l=l(G,t)$ are defined in \href{def s} and \href{def l}.
\end{theorem}
\begin{proof}
Assume to the contrary that there exists some vertex $v \in V(G')$ with $\phi(v) \geq n + 3t + s - l + 1$. %Let $B' = \{u \in V(G') : \phi(u) \geq m+2t+1\}$. 
By \Href{main  lemma}, $s \geq l$. From the calculation of \href{sl block} the following contradiction emerges (recall $m' = m+2t$):
$$
\begin{array}{llll}  
s 
&=& \sum\limits_{v \in B'} \phi(v) - \sum\limits_{e \in R} f(e) \\
%&=& \sum\limits_{v \in B} \phi(v) + \sum\limits_{i = |W| + 1}^{t} \phi(w_i)  - \sum\limits_{e \in R} f(e) \\    
%&=& \sum\limits_{v \in B'} \phi(v) - \sum\limits_{e \in R} f(e) \\
&\geq& 
(n+3t+s-l+1)+\sum\limits_{i=1}^{t+n-m-1} (m'+i) + |R|(m'+1) - |R|m'\\
&\geq& (n+3t+s-l+1) + l -(m+2t+t+n-m) \\
&=& s+1.
\end{array}
$$
Thus, the proof is complete. 
\end{proof}

%\liu{Below, change $G$ to $G'$  and $m$ to $m'$} 
\begin{lemma}
\label{two largest labels general}
Let $p$ be a positive integer. Let $G'$ be an $m'$-edge graph that has a degree-2 vertex $v$ which is incident to edges $e_{m'}$ and $e_{m'-1}$. Let $G^*$ be the graph obtained by subdividing $e_{m'}$ into $p$ edges. If $G'$ admits an antimagic labeling $f$ such that $f(e_{m'}) = m'$, $f(e_{m'-1}) = m' - 1$, and $\phi_{f}(u) \leq 2m' -1$ for all $u \in V(G')$,  %\max\{\phi(u) \ | \ u \in V(G)\}$, 
then there exists an antimagic labeling for $G^*$.
\end{lemma}
\begin{proof}
We prove the result by induction on $p$. Assume $p=2$. Let $f$ be an antimagic labeling of $G'$ such that $f(e_m') = m'$, $f(e_{m'-1}) = m' - 1$, and $\phi(v) = 2m'-1 = \max\{\phi(u) \ | \ u \in V(G')\}$. Let $G^*$ be obtained from $G'$ by subdividing $e_m$ into two edges, called $e_{m'}$ and $e_{m'+1}$, where  $e_{m'+1}$ is incident to $e_{m'-1}$. 
Let $f'$ be a labeling for $G^*$ defined by $f'(e) = f(e)$ if $e \neq e_{m'+1}$, and $f'(e_{m'+1}) = m'+1$. By the assumption, all vertices $u$ not incident to $e_{m'+1}$ have $\phi_{f'}(u) \leq 2m'-1$, while the two vertices incident to $e_{m'+1}$ have distinct vertex sums and both are greater than $2m'-1$. Thus, $f'$ is an antimagic labeling for $G^*$. 

Furthermore, under $f'$, the vertex incident to $e_{m'+1}$ and $e_{m'}$ is a degree-2 vertex in $G^*$ which is incident to the largest labels $m'$ and $m'+1$ and has the maximum $\phi$-value. Therefore, the result follows by induction on $p$. 
\end{proof}

\begin{lemma}
\label{lemma am general} 
Let $p \geq 2$ and $t \geq 0$ be integers. Let $G$ be an $m$-edge $n$-vertex graph that has a degree-2 vertex $v$, where $e_m$ and $e_{m-1}$ are the edges incident to $v$. Let $s=s(G, t)$ and $l=l(G, t)$, as defined in \href{def s} and \href{def l}.  Let $G^*$ be the graph obtained by sub-diving $e_m$ into $p$ edges. If $m \geq n$ and $G \cup t P_3$ admits an antimagic labeling $f$ such that $f(e_m), f(e_{m-1}) \geq t + s - l$,   
%and $f(e_{m-1}) \geq t + s - l$, 
then $G^* \cup t P_3$ is antimagic.
\end{lemma}
\begin{proof}
Assume $p = 2$. Let $f$ be an antimagic labeling of $G \cup tP_3$ such that $f(e_m), f(e_{m-1}) \geq t + s - l$. 
%and $f(e_{m-1}) \geq t + s - l$. 
Let $G^*$ be obtained from $G$ by subdividing $e_m$ into two edges, $e_m$ and $e_{m+1}$, where $e_{m+1}$ is incident to $e_{m-1}$. 
%; all other edges in $G$ remain the same. 
%\mason{maybe this should be f*} 
Let $f^*$ be a labeling for $G^* \cup tP_3$ defined by $f^*(e) = f(e)$ if $e \neq e_{m+1}$,  and $f^*(e_{m+1}) = m + 2t + 1$. By \Href{max phi bound} all vertices $u$ not incident to $e_{m+1}$ have $\phi(u) \leq n + 3t + s - l$. Since $m \geq n$, the two degree-2 vertices $v$ and $w$ incident to $e_{m+1}$ have $\phi (v), \phi (w) > n + 3t + s - l$ and $\phi(v) \neq \phi(w)$. %(since $\deg(v)=2$). 

Assume $p = 3$. Let $G^{**}$ be obtained from $G^*$ by subdividing $e_{m+1}$ into two edges, $e_{m+1}$ and $e_{m+2}$, where $e_{m+2}$ is  incident to $e_{m-1}$. 
%; all other edges remain the same. 
Let $f^{**}$ be a labeling for $G^{**}$ defined by $f^{**}(e) = f^*(e)$ if $e \neq e_{m+2}$, and $f^{**}(e_{m+2}) = m + 2t + 2$. Similar to the above, it is not difficult to show that $f^{**}$ is an antimagic labeling for $G^{**}$. In addition, under $f^{**}$ the degree-2 vertex incident to $e_{m+2}$ and $e_{m+1}$ is incident to the largest labels, $m + 2t + 1$ and $m + 2t + 2$, and has the maximum $\phi$-value.  
By applying \Href{two largest labels general} to $G^{**}$, \Href{lemma am general} follows. 
\end{proof}

%\liu{Added the paragraph below and moved Lemma 18 here.} 

After establishing the above two recursive results for  general graphs, in the remaining of this section we shall focus on 2-regular graphs.  Denote a cycle $C_n$ by  $V(C_n)=\{v_1, \ldots, v_{n}\}$ and $E(C_n) = \{e_1, e_2, \ldots, e_n\}$, where $e_i=v_iv_{i+1}$ for $1 \leq i \leq n-1$, and $e_n=v_{n}v_1$.

\begin{lemma}
\label{3-cycles}
There exist antimagic labelings for $C_n \cup tP_3$, $0 \leq t \leq 6$ and $n\geq3$, such that the two largest labels, $2t+n$ and $2t+n-1$, are assigned to incident edges on $C_n$. 
Consequently, for all $n \geq 3$, $\tau(C_n) \geq 6$.
\label{basecase of lemma 9}
\end{lemma}
\begin{proof}
By  \Href{two largest labels general}, it suffices to show the existence of antimagic labelings for $C_3 \cup tP_3$, $0 \leq t \leq 6$, such that the two largest labels, $2t+3$ and $2t+2$, are assigned to incident edges on $C_3$. Such labelings are given in \Href{c3 labeling}.  In the table, for each $t$, the labeling $f_t$ consists of a 3-tuple $(f_t(e_1)$,  $f_t(e_2)$, $f_t(e_3))$ for  $E(C_3)$ (where the two largest labels, $2t+3$ and $2t+2$, are underlined) and  $t$ pairs of labels for the  $3$-paths. An example is illustrated in \Href{fig:ex c3+4p3}. %By \Href{two largest labels general}, these labelings can be extended to all $n\geq3$
\end{proof}

\begin{table}[hbtp!]
\centering
\fontsize{10pt}{9pt}\selectfont
\begin{tabular}{ |c|p{4.5cm}|l| } 
\hline
\multirow{2}{*}{$t$} &$f_t(e_1, e_2, e_3)$, $\{\phi(v_1), \phi(v_2), \phi(v_3)\}$& \multicolumn{1}{c|}{\multirow{2}{*}{Pairs of labels on $P_3$ with  their sums}} \\ \hline
0 & (1, \underline{2}, \underline{3}), \{3, 4, 5\} & \\ \cline{1-3}
1 & (3, \underline{4}, \underline{5}), \{7, 8, 9\} & (1, 2 $|$ 3)  \\ \hline
2 & (4, \underline{6}, \underline{7}), \{10, 11, 13\} & (1, 3 $|$ 4) (2, 5 $|$ 7)  \\ \hline
3 & (5, \underline{8}, \underline{9}), \{13, 14, 17\} & (1, 4 $|$ 5) (2, 6 $|$ 8) (3, 7 $|$ 10)  \\ \hline
4 & (6, \underline{10}, \underline{11}), \{16, 17, 21\} & (1, 5 $|$ 6) (2, 8 $|$ 10) (4, 7 $|$ 11) (3, 9 $|$ 12)  \\ \hline
5 & (6, \underline{12}, \underline{13}), \{18, 19, 25\} & (1, 5 $|$ 6) (2, 10 $|$ 12) (4, 9 $|$ 13) (3, 11 $|$ 14) (7, 8 $|$ 15)  \\ \hline
6 & (3, \underline{14}, \underline{15}), \{17, 18, 29\} & (1, 2 $|$ 3) (4, 10 $|$ 14) (6, 9 $|$ 15) (5, 11 $|$ 16) (7, 12 $|$ 19) (8, 13 $|$ 21)  \\ \hline
\end{tabular}
\caption{Antimagic labelings for $C_3 \cup tP_3$, $0 \leq t \leq 6$.}
\label{c3 labeling}
\end{table}

\begin{figure} 
    \centering
    \begin{tikzpicture}
        % 6 & (3, \underline{14}, \underline{15}), \{17, 18, 29\} & (1, 2 $|$ 3) (4, 10 $|$ 14) (6, 9 $|$ 15) (5, 11 $|$ 16) (7, 12 $|$ 19) (8, 13 $|$ 21) % data from table
        % $C_3$ and its labels
        \begin{scope}[shift={(-3,-0.55)}, scale=1.5]
            \newcommand{\ls}{0.15} % size of circles for phi-values
            \newcommand{\s}{0.08}
	        % this ox needs to be lowered a little so it can look better
            \newcommand{\ox}{0.3} % off set x variable to make circles 
        
            \draw (1,0.5) -- (0,1) -- (0,0) -- (1,0.5);

            \draw[fill=black] (0, 0) circle (\s);
	        \draw[fill=black] (1, 0.5) circle (\s);
	        \draw[fill=black] (0, 1) circle (\s);
	        
	        \footnotesize
	        \draw (0, 0 - \ox) circle (\ls) node {17};
            \draw (1 + \ox, 0.5) circle (\ls) node {18};
            \draw (0, 1 + \ox) circle (\ls) node {29};
	    
	   	    \node[left]  at (0, 0.5) {\underline{14}};
            \node[above] at (0.5, 0.8) {\underline{15}};
            \node[below] at (0.5, 0.2) {$3$};
        \end{scope}
        
        \newcommand{\s}{0.09} % size of vertices
        \newcommand{\ls}{0.25} % size of circles for phi-values
        % this ox needs to be lowered a little so it can look better
        \newcommand{\ox}{0.45} % off set x variable to make circles 
        % \o\u\l\p denote o: offset (which path is being drawn); u: upper label; l: lower label; p: middle phi value, 
        \newcounter{o} % this counter is the offset for each path.
        \foreach \u\l\p in {1/2/3, 4/10/14, 6/9/15, 5/11/16, 7/12/19, 8/13/21}{
	        \draw(1.5*\theo, -1) -- (1.5*\theo,1);
	        \draw[fill=black] (1.5*\theo, 1) circle (\s);
	        \draw[fill=black] (1.5*\theo, 0) circle (\s);
	        \draw[fill=black] (1.5*\theo,-1) circle (\s);
	        
	        \footnotesize
	        \draw (1.5*\theo + \ox, 1) circle (\ls) node {\u};
	        \draw (1.5*\theo + \ox, 0) circle (\ls) node {\p};
	        \draw (1.5*\theo + \ox, -1) circle (\ls) node {\l};
	        \node[left] at (1.5*\theo, 0.5) {\u};
	        \node[left] at (1.5*\theo,-0.5) {\l};
	        
	        \stepcounter{o}
	        %\addtocounter{o}{2}
	    }
	    \setcounter{o}{0} % counters are global vars, they need to be reset after use.
	\end{tikzpicture}
    \caption{An antimagic edge labeling of $C_3 \cup 6P_3$. Circled numbers are  $\phi$-values while other  numbers are edge labels. The two largest labels are underlined.}
    \label{fig:ex c3+4p3}
\end{figure}
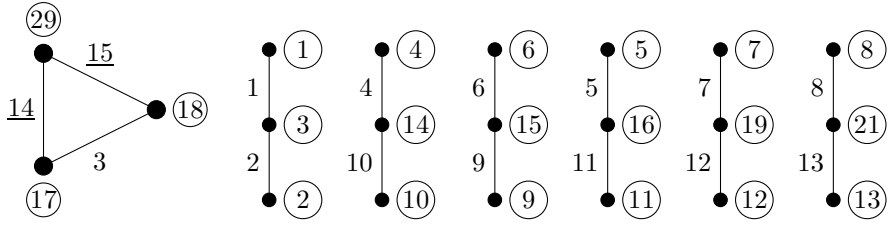

%\liu{Modified below}

\begin{comment}
\noindent
\angel{old remark}{\bf Remark.}  Consider the graph $G \cup tP_3$ where $G$ is a 2-regular graph with $n$ vertices (so $n$ edges). Then $|E(G \cup tP_3)|=m=n+2t$. Suppose the assumptions  of \Href{lemma am general} hold for $G \cup tP_3$, that is, there exist an antimagic labeling $f$ for $G \cup tP_3$ and an edge $e \in E(G)$ with $ t + s - l + 1 \leq f(e) \leq n+2t$, where $s=s(G, t)$ and $l=l(G, t)$.    Hence, we obtain  
$t + s - l + 1 \leq n+2t$, 
a quadratic inequality in $t$ with a positive solution:

\begin{align}
t & \geq n - \frac{1}{2} + \sqrt{(\sqrt{2}n - \frac{1}{\sqrt{2}})^2 + \frac{7}{4}} \label{lower bound of tau} \\
  & = (1 + \sqrt{2})(n - \frac{1}{2}) + O(\frac{1}{n}). 
\label{lower bound of tau appx}
\end{align}

\noindent
Further, with $G = H \cup tP_3$ where $H$ is a 2-regular graph with $n$ vertices, the assumptions of \Href{two largest labels general} imply the existence of a  vertex $v$ with $\phi(v) = 2n+4t-1$. By \Href{max phi bound}, it must be that $2n+4t-1 \leq n + 3t + s - l$, equivalently $t + s - l + 1 \geq n + 2t$, the reverse of \href{lower bound of tau}.  Thus for a graph $C_n \cup tP_3$, if \href{lower bound of tau appx} is strict then it is impossible to find a labeling satisfying the conditions of \Href{two largest labels general}.

\end{comment} 

\noindent
{\bf Remark.}  Consider the graph $G \cup tP_3$ where $G$ is a 2-regular graph with $n$ vertices. Then $|E(G)| = n$ and $m' = |E(G \cup tP_3)| = n+2t$. Suppose the assumptions  of \Href{two largest labels general} hold for $G \cup tP_3$, that is, there exists an antimagic labeling $f$ for $G \cup tP_3$ and a vertex $v \in V(G)$ with $\phi(v) = 2m'-1 = 2n+4t-1$. By \Href{max phi bound},  $\phi(v) \leq n+3t+s-l$ where $s=s(G, t)$ and $l=l(G, t)$. Hence, we obtain  
$s - l + 1 \geq n+t$, 
a quadratic inequality in $t$ with a positive solution:
\begin{align}
t & \leq n - \frac{1}{2} + \sqrt{(\sqrt{2}n - \frac{1}{\sqrt{2}})^2 + \frac{7}{4}} \label{lower bound of tau} \\
  & = (1 + \sqrt{2})(n - \frac{1}{2}) + O(\frac{1}{n}). 
\label{lower bound of tau appx}
\end{align}
Now suppose the assumptions  of \Href{lemma am general} hold for $G \cup tP_3$. Then there exist an antimagic labeling $f$ for $G \cup tP_3$ and an edge $e \in E(G)$ with $ t + s - l + 1 \leq f(e) \leq n+2t$, where $s=s(G, t)$ and $l=l(G, t)$. We then obtain  
$t + s - l + 1 \leq n+2t$, the reverse of \href{lower bound of tau}. 
Note that when the equality in \href{lower bound of tau} holds then $m' = t+s-l+1$, implying the assumptions of \Href{lemma am general} and \Href{two largest labels general} are equivalent.  
In conclusion, for a 2-regular graph $G$, to investigate possible $t$ values, $0 \leq t \leq \beta(G) = (1+\sqrt{2})(n + \frac{1}{2})$, one might consider the following two sub-intervals:   
$$
[0, (1+\sqrt{2})(n + \frac{1}{2})] = 
[0, (1+\sqrt{2})(n - \frac{1}{2})] \cup [(1+\sqrt{2})(n - \frac{1}{2}), (1+\sqrt{2})(n + \frac{1}{2})].
$$
In the first sub-interval, $0 \leq t  \leq (1+\sqrt{2})(n - \frac{1}{2})$, it might be possible to find a labeling satisfying the conditions of \Href{two largest labels general} (i.e., the largest two labels are assigned to incident two edges). Likewise for the second sub-interval it might be possible to find a labeling satisfying the conditions of \Href{lemma am general}. 
%Note, when we have equality in \href{lower bound of tau}, any labeling $f$ of $G \cup tP_3$ satisfying the assumptions of \Href{lemma am general} must also be satisfying the assumptions of \Href{two largest labels general}. 
%In which only one of \Href{two largest labels general} or \Href{lemma am general} can be applied. 
%\angel{end of new remark} \angel{I think in my last sentence i implied existence of a labeling which will allow us to actually use the lemmas. what do you all think.}

From the above discussion, for a 2-regular graph $G$, if there exists  an antimagic labeling $f$ for $G \cup tP_3$ satisfying the assumptions of \Href{lemma am general}, then $t \geq (1+\sqrt{2}) (n-\frac{1}{2})$. In the following result we prove that the converse of this also holds for small values of $n$.

\begin{lemma}
\label{labelings for cycles}
Let $n$ and $t$ be integers such that $3 \leq n \leq 9$ and 
$$(1 + \sqrt{2})(n - \frac{1}{2}) \leq t \leq (1 + \sqrt{2})(n + \frac{1}{2}).
$$
Then $C_n \cup tP_3$ admits an antimagic  labeling such that there exist two incident edges on $C_n$ receiving labels that are at least $t+s-l$, where $s$ and $l$ are defined in \href{def s} and \href{def l}, respectively.
\end{lemma}
\begin{proof} 
%\liu{Added the following} 
%By \Href{cycle}, $t \leq (1 + \sqrt{2})(n + \frac{1}{2})$. 
For $3 \leq n \leq 9$ and $(1 + \sqrt{2})(n - \frac{1}{2}) \leq t \leq (1 + \sqrt{2})(n + \frac{1}{2})$, we give a list of such labelings in \Href{am labeling} 
in the Appendix. These labelings were constructed with computer assistance. For each $n$ and $t$, the labeling $f_t$ consists of an $n$-tuple $(f_t(e_1)$, $f_t(e_2)$, $\cdots$, $f_t(e_n))$ for $E(C_n)$ (where two adjacent labels greater than or equal to $t+s-l$ are underlined) and $t$ pairs of labels for the 3-paths. See \Href{fig:ex c5+13p3} as an example.
\end{proof}

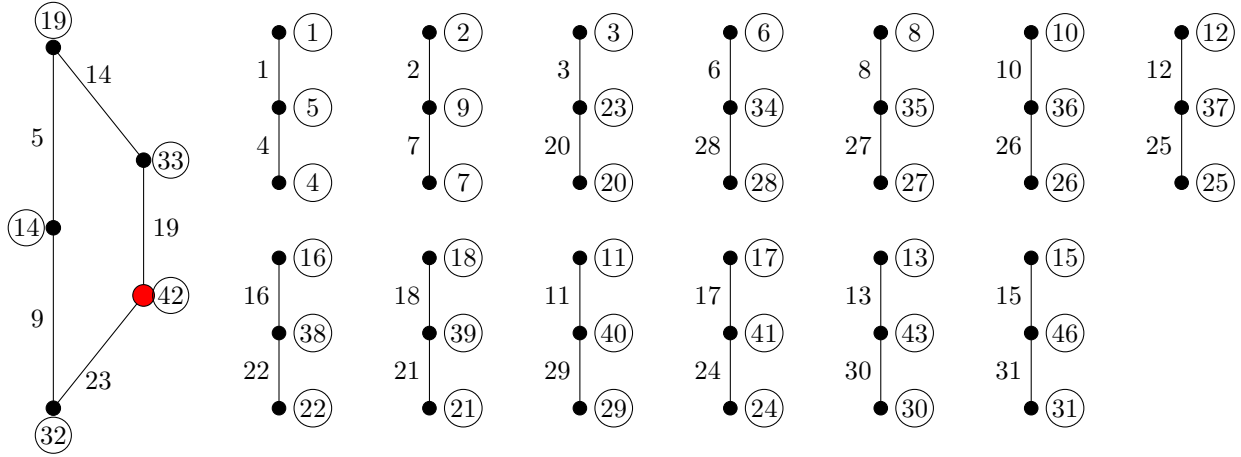
\begin{figure}
    \centering
    \begin{tikzpicture}
        % 5 & 13 & (9, 5, 14, 19, 23) & (5, 19, 14) (17, 24, \underline{32}), \{19, 24, 33, 41, 49, 56\} &(17, 24) & (9, 23) % data from old table 
        % drawing the C_5
        \begin{scope}[shift={(-3,-4)}, scale=1.2]
            \newcommand{\ls}{0.2} % size of circles for phi-values
            \newcommand{\s}{0.08} % size of vertices
	        % this ox needs to be lowered a little so it can look better
            \newcommand{\ox}{0.3} % off set x variable to make circles 
        
            \draw (0,0)-- (1,1.25) -- (1,2.75) -- (0,4)-- (0,2) -- (0,0);

            \draw[fill=black] (0, 0) circle (\s);
	        \draw[fill=red] (1, 1.25) circle (1.5*\s);
	        \draw[fill=black] (1, 2.75) circle (\s);
	        \draw[fill=black] (0, 4) circle (\s);
	        \draw[fill=black] (0,2) circle (\s);
	        
	        \footnotesize
	        \draw (0, 0 - \ox) circle (\ls) node {32};
            \draw (1 + \ox, 1.25) circle (\ls) node {42};
            \draw (- \ox, 2 ) circle (\ls) node {14};
            \draw (0, 4 + \ox) circle (\ls) node {19};
            \draw (1+ \ox,2.75 ) circle (\ls) node {33};
	    
	   	    \node[left]  at (0, 1) {9};
	   	    \node[left]  at (0, 3) {5};
	   	     \node[above] at (0.5, 3.5){14};
            \node[right]  at (1, 2) {19};
            \node[below] at (0.5, 0.5) {23};
        \end{scope}
     
        \newcommand{\s}{0.09} % size of vertices
        \newcommand{\ls}{0.25} % size of circles for phi-values
        % this ox needs to be lowered a little so it can look better
        \newcommand{\ox}{0.45} % off set x variable to make circles 
        % \u\l\p denote u: upper label; l: lower label; p: middle phi value,
   %     \newcounter{o} % this counter is the x-offset, it resets to 0 after drawing 7 paths.
        \newcounter{oy} % this counter is the y-offset it only increases after 7 paths are drawn
        \foreach \u\l\p in {1/4/5, 2/7/9, 3/20/23, 6/28/34, 8/27/35, 10/26/36, 12/25/37, 16/22/38, 18/21/39, 11/29/40, 17/24/41, 13/30/43, 15/31/46}{
	        \draw(\theo, -1 + \theoy) -- (\theo, 1 + \theoy);
	        \draw[fill=black] (\theo, 1 + \theoy) circle (\s);
	        \draw[fill=black] (\theo, 0 + \theoy) circle (\s);
	        \draw[fill=black] (\theo,-1 + \theoy) circle (\s);
	        
	        \footnotesize
	        \draw (\theo + \ox, 1 + \theoy) circle (\ls) node {\u};
	        \draw (\theo + \ox, 0 + \theoy) circle (\ls) node {\p};
	        \draw (\theo + \ox, -1 + \theoy) circle (\ls) node {\l};
	        \node[left] at (\theo, 0.5 + \theoy) {\u};
	        \node[left] at (\theo,-0.5 + \theoy) {\l};
	        
	        \addtocounter{o}{2}
	        \ifnum\theo=14
	            % this if-statement is the one that checks that 7 paths have been drawn 
	            % the next path will be drawn below and to the left
	            \addtocounter{oy}{-3}
	            \setcounter{o}{0}
	        \fi
	    }
	    
	\end{tikzpicture}
    \caption{An antimagic edge labeling of $C_5 \cup 13P_3$. The vertex satisfying the hypotheses of \Href{lemma am general} is bolded and colored red.}
    \label{fig:ex c5+13p3}
\end{figure}

Combining 
%\Href{two largest labels general}, 
\Href{lemma am general}, \Href{basecase of lemma 9}, and \Href{labelings for cycles}, we obtain: 

\begin{corollary}
\label{20}
Let $n$ and $t$ be integers such that $t \leq \min\{22, \lfloor (1 + \sqrt{2})(n + \frac12) \rfloor\}$. 
Then $C_n \cup t P_3$ is antimagic.
\end{corollary}

As  $\beta(C_9) =22$, \Href{20} implies:
  
\begin{theorem}
\label{cycles} 
For $3 \leq n \leq 9$,  $\tau(C_n) = \beta(C_n)$.  
Equivalently, for $3 \leq n \leq 9$, $C_n \cup tP_3$ is antimagic if and only if 
$$
0 \leq t \leq \left\lfloor (1 +  \sqrt{2})(n + \frac{1 }{2})    \right\rfloor.  
$$
\end{theorem}
%\mason{do we need to have a transition sentence here?}

We conjecture that the result of \Href{cycles} holds for all $n$. 

\begin{conjecture}
\label{tau beta for Cn}
For any $n$, $\tau(C_n)=\beta(C_n)$.
\end{conjecture}

\noindent 
To confirm \Href{tau beta for Cn} for $n \geq 10$, it is sufficient to extend \Href{labelings for cycles} for all $n$.

We conclude this section with the following two results. 
\begin{theorem}
\label{parker} 
Let $G$ be a 2-regular graph. If $G \cup tP_3$ is antimagic, then $C_q \cup G \cup tP_3$ is also antimagic for any $q \geq 3$.
\end{theorem}
\begin{proof}
Suppose $q=3$ and let $G$ be a 2-regular graph with $n$ vertices. Let $f$ be an antimagic labeling of $G \cup tP_3$.  
Since the degree of every vertex $v$ in $V(G \cup tP_3)$ is at most 2, and $m'=|E(G \cup tP_3)|=n+2t$,  we have $\phi(v) \leq 2n+4t-1$.
We extend $f$ to $C_3 \cup G \cup tP_3$ by assigning the largest three labels, $n+2t+i$, $1 \leq i \leq 3$, to $E(C_3)$ and keeping the same   labels for other edges. Then we have  $\phi(v) > 2n+4t$ for every $v \in V(C_3)$.  Therefore, $f$ is an antimagic labeling for $C_3 \cup G \cup tP_3$. As the largest two labels are assigned to incident edges on $C_3$, the result for $q \geq 4$  follows by  \Href{two largest labels general}. 
\end{proof}

\begin{theorem}
\label{double c3}
For any $n \geq 3$, $\tau(C_n \cup C_3) \geq 15$. Moreover, 
$\tau(2C_3)=\beta(2C_3)=15$. 
\end{theorem}

\begin{proof} 
Combining \Href{parker} with the antimagic labelings for $C_3 \cup tP_3$, $0 \leq t \leq 8$, given in \Href{c3 labeling}, we obtain   $\tau(C_n \cup C_3) \geq 8$ for $n \geq 3$. In addition, \Href{tab:doube c3 B} gives antimagic labelings for $2C_3 \cup tP_3$ for $9 \leq t \leq 15$, where  the labelings 
satisfy  the following properties:
\begin{itemize}
    \item For each $9 \leq t \leq 13$, the labeling assigns the two largest labels, $m'$ and $m'-1$, to incident edges on $C_3$, satisfying the hypotheses of \Href{two largest labels general}.
    \item For each $t=14,15$, the labeling assigns two labels  greater than or equal to the value of $t+s-l$ to incident edges  on $C_3$,  satisfying the hypotheses of \Href{lemma am general}.
\end{itemize}

\noindent
Thus, $C_n \cup C_3 \cup tP_3$ is antimagic for $n \geq 3$ and $9 \leq t \leq 15$, implying $\tau(C_n \cup C_3) \geq 15$ for $n \geq 3$. 
The ``moreover'' part follows by \Href{cycle}, which shows $\tau(2C_3)  \leq 15$. 
\end{proof}

See \Href{fig:ex 2c3+13p3} for an example.

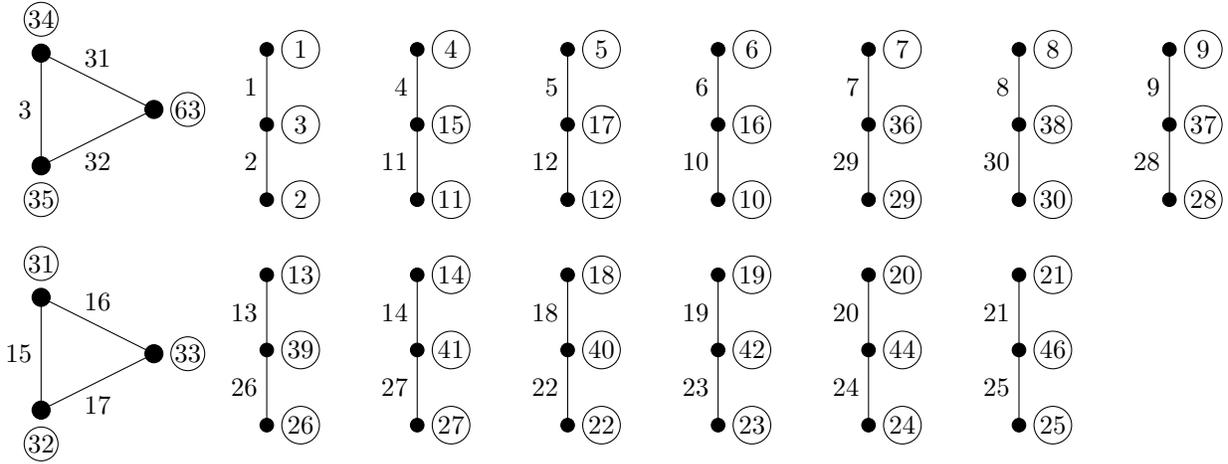
\begin{figure}[h]
    \centering
    \begin{tikzpicture}
        % $C_3$ and its labels
        % 5 & 13 & (9, 5, 14, 19, 23) & (5, 19, 14) (17, 24, \underline{32}), \{19, 24, 33, 41, 49, 56\} &(17, 24) & (9, 23)
        % new one, mistakes in old one
        % (5, 9, 14) (19, 23, 29), \{14, 19, 23, 42, 48, 52\}
        \begin{scope}[shift={(-3,-3.8)}, scale=1.5]
            \newcommand{\ls}{0.15} % size of circles for phi-values
            \newcommand{\s}{0.08}
	        % this ox needs to be lowered a little so it can look better
            \newcommand{\ox}{0.3} % off set x variable to make circles 
        
            \draw (1,0.5) -- (0,1) -- (0,0) -- (1,0.5);

            \draw[fill=black] (0, 0) circle (\s);
	        \draw[fill=black] (1, 0.5) circle (\s);
	        \draw[fill=black] (0, 1) circle (\s);
	        
	        \footnotesize
	        % phi values
	        \draw (0, 0 - \ox) circle (\ls) node {32};
            \draw (1 + \ox, 0.5) circle (\ls) node {33};
            \draw (0, 1 + \ox) circle (\ls) node {31};
	        % edges
	   	    \node[left]  at (0, 0.5) {15};
            \node[above] at (0.5, 0.8) {16};
            \node[below] at (0.5, 0.2) {17};
        \end{scope}
        \begin{scope}[shift={(-3,-0.55)}, scale=1.5]
            \newcommand{\ls}{0.15} % size of circles for phi-values
            \newcommand{\s}{0.08}
	        % this ox needs to be lowered a little so it can look better
            \newcommand{\ox}{0.3} % off set x variable to make circles 
        
            \draw (1,0.5) -- (0,1) -- (0,0) -- (1,0.5);

            \draw[fill=black] (0, 0) circle (\s);
	        \draw[fill=black] (1, 0.5) circle (\s);
	        \draw[fill=black] (0, 1) circle (\s);
	        
	        \footnotesize
	        \draw (0, 0 - \ox) circle (\ls) node {35};
            \draw (1 + \ox, 0.5) circle (\ls) node {63};
            \draw (0, 1 + \ox) circle (\ls) node {34};
	    
	   	    \node[left]  at (0, 0.5) {3};
            \node[above] at (0.5, 0.8) {31};
            \node[below] at (0.5, 0.2) {32};
        \end{scope}
        
        \newcommand{\s}{0.09} % size of vertices
        \newcommand{\ls}{0.25} % size of circles for phi-values
        % this ox needs to be lowered a little so it can look better
        \newcommand{\ox}{0.45} % off set x variable to make circles 
        % \o\u\l\p denote o: offset (which path is being drawn); u: upper label; l: lower label; p: middle phi value,
     %   \newcounter{o} % this counter is the x-offset, it resets to 0 after drawing 7 paths.
        \setcounter{o}{0}
    %    \newcounter{oy} % this counter is the y-offset it only increases after 7 paths are drawn
        \setcounter{oy}{0}
        %(18, 21 $|$ 39) (11,  29 $|$ 40) & (11, 18 $|$ 29) (32, 21 $|$ 53) 
        \foreach \u\l\p in {1/2/3, 4/11/15, 5/12/17, 6/10/16, 7/29/36, 8/30/38, 9/28/37, 13/26/39, 14/27/41, 18/22/40, 19/23/42, 20/24/44, 21/25/46}{
	        \draw(\theo, -1 + \theoy) -- (\theo, 1 + \theoy);
	        \draw[fill=black] (\theo, 1 + \theoy) circle (\s);
	        \draw[fill=black] (\theo, 0 + \theoy) circle (\s);
	        \draw[fill=black] (\theo,-1 + \theoy) circle (\s);
	        
	        \footnotesize
	        \draw (\theo + \ox, 1 + \theoy) circle (\ls) node {\u};
	        \draw (\theo + \ox, 0 + \theoy) circle (\ls) node {\p};
	        \draw (\theo + \ox, -1 + \theoy) circle (\ls) node {\l};
	        \node[left] at (\theo, 0.5 + \theoy) {\u};
	        \node[left] at (\theo,-0.5 + \theoy) {\l};
	        
	        \addtocounter{o}{2}
	        \ifnum\theo=14
	            % this if-statement is the one that checks that 7 paths have been drawn 
	            \addtocounter{oy}{-3}
	            \setcounter{o}{0}
	        \fi
	    }
	    % i want to dash the replaced path.
	   % \draw[dashed] (5, -4.5) rectangle (1.3,-1.5); 
	\end{tikzpicture}
    \caption{An antimagic labeling of $2C_3 \cup 13P_3$.} 
    %The labels of the $P_3$ components are the same as the ones in $C_5 \cup 13P_3$ (see \Href{fig:ex c5+13p3}) except that the $P_3$ components labeled (18, 21 $|$ 39) and (11, 29 $|$ 40) are replaced by (11, 18 $|$ 29) and (32, 21 $|$ 53) in the dashed rectangle.}
    \label{fig:ex 2c3+13p3}
\end{figure}
%%%%%%%%%%%%%%%%%%%%%%
%
\section{Open Problems and Future Work} \label{problems}
%
%%%%%%%%%%%%%%%%%%%%%

To study the general properties of $\tau(G)$, the triangle inequality emerges: 

\begin{conjecture}
\label{triangle inequality}
For any antimagic graphs $G$ and $H$, it holds  that $\tau(G \cup H) \leq \tau(G) + \tau(H)$.
\end{conjecture}
\noindent
We partially confirm  \Href{triangle inequality} with additional conditions. 
Let $G$ be a graph with $m_G$ edges and $n_G$ vertices. Denote the first bound in \Href{general bound} by $b^1_G = (3 + 2\sqrt{2})(m_G - n_G) + (1 + \sqrt{2})(m_G + 1/2)$, and define  
$\Delta^1_{G} = b^1_{G} - \tau(G)$.  
If $G$ and $H$ are graphs with $\Delta^1_G + \Delta^1_H \leq \frac{1+\sqrt{2}}{2}$, then 
$\tau(G \cup H) \leq \tau(G) + \tau(H)$. By    \Href{general bound},  
$$
\begin{array}{llll}
\tau(G \cup H) &\leq& b^1_{G \cup H}  \\ 
&=& b^1_G + b^1_H - \frac{1+\sqrt{2}}{2} \ \  \  \ (\mbox{$\because$ $m_{G \cup H} = m_G + m_H$ and $n_{G \cup H} = n_G + n_H$}) \\
&=& \tau(G) + \Delta^1_G + \tau(H) + \Delta^1_H - \frac{1+\sqrt{2}}{2} \\ 
&\leq& \tau(G) + \tau(H).  \ \ \ \mbox{($\because$ \ $\Delta^1_G + \Delta^1_H \leq \frac{1+\sqrt{2}}{2}$})  
\end{array} 
$$
A similar statement can be made using the second bound of \Href{general bound}. Let $G$ be a graph with $n_G$ vertices, $m_G$ edges, $k_G$ internal edges, and $t'_G$ components isomorphic to $P_3$. Denote the second bound of \Href{general bound} by $b^2_G = 2m_G + 5(k_G-t'_G) + 1$, and denote  
$\Delta^2_{G} = b^2_{G} - \tau(G)$. 
If $G$ and $H$ are graphs with $\Delta^2_G + \Delta^2_H \leq 1$, then 
$\tau(G \cup H) \leq \tau(G) + \tau(H)$. 

For a graph $G$, we have established in \Href{general bound} upper bounds $\beta(G)$ for  $\tau(G)$, and  we have shown that many graphs have $\tau(G)=\beta(G)$.  On the other hand, if there exists a graph $G$ satisfying the hypotheses in \Href{general bound} but has $\tau(G) < \beta(G)$,  then for any positive integer $t$, where $\tau(G) < t \leq \beta(G)$, it remains to determine whether $G \cup tP_3$ is antimagic or not. We conjecture that $G \cup tP_3$ is not antimagic for those values of $t$ if they exist:   

\begin{conjecture}
\label{tau is Bernoulli}
For a graph $G$, $G \cup tP_3$ is antimagic if and only if $t \leq \tau(G)$. That is, $\tau(G)$ is the maximum integer $t$ such that $G \cup tP_3$ is antimagic.
\end{conjecture}

\noindent
\Href{tau is Bernoulli} is equivalent to the following statement:  
If $G \cup tP_3$ is antimagic for some positive integer $t$,  then $G \cup (t-1)P_3$ is antimagic.

Thus far, we have not found a graph with $\tau(G) \neq \beta(G)$. Therefore, we ask if our bound is always tight:

\begin{question}\label{bigquestion}
Does there exist a graph $G$ without isolated vertices nor $P_2$ as components with  $\beta(G)\geq 0$ and
$\tau(G) < \beta(G)$? %\mason{can we say this without using "nor"}
\end{question}
\noindent
With computer aid, we checked various small graphs, and found the following graphs have  $\tau(G)= \beta(G)$: the clique $K_4$, the graph induced by deleting an edge from $K_4$, and $C_3$ with at most one pendant edge extended from each of three the vertices on $C_3$.

Because $\beta(G)\geq 0$ when every component of $G$ has at least 3 edges, a weaker question arises:

\begin{question}
\label{masonsquestion}
Are all graphs without isolated vertices nor $P_2$ nor $P_3$ as components antimagic?
\end{question}

\noindent
An affirmative answer to \Href{masonsquestion} or a negative answer to \Href{bigquestion} would imply that \Href{conj} is true.

\bigskip

\noindent
{\bf Acknowledgement.} The authors would like to thank the three anonymous  referees for their careful reading of the manuscript and for their  insightful suggestions.
%which led to a better presentation of the article. 

%\end{document}

%\section*{Appendix}

\newpage

\noindent
{\Large\bf Appendix}
{\fontsize{10pt}{9pt}\selectfont %\allowdisplaybreaks
\begin{longtable}{|p{0.2cm}|p{0.2cm}|c|c|p{2.7cm}|p{10cm}| }\hline
\multirow{2}{*}{$n$} & \multirow{2}{*}{$t$} & \multirow{2}{*}{$m'$} &\multirow{2}{*}{$t+s-l$}&$f_t(e_1, e_2, \cdots, e_n)$ $\{\phi(V(C_n))\}$&  \multicolumn{1}{c|}{\multirow{2}{*}{Pairs of labels on $P_3$ and their sums}} \\ \hline
\multirow{4}{*}{3} & \multirow{2}{*}{7} & \multirow{2}{*}{17} & \multirow{2}{*}{13} & (6, \underline{13}, \underline{17}), \{19, 23, 30\} & (2, 4 $|$ 6) (1, 12 $|$ 13) (7, 10 $|$ 17) (3, 15 $|$ 18) (9, 11 $|$ 20) (5, 16 $|$ 21) (8, 14 $|$ 22) \\ \cline{2-6}
 & \multirow{2}{*}{8} & \multirow{2}{*}{19} & \multirow{2}{*}{10} & (11, \underline{14}, \underline{15}), \{25, 26, 29\} & (1, 10 $|$ 11) (2, 12 $|$ 14) (6, 9 $|$ 15) (7, 13 $|$ 20) (3, 18 $|$ 21) (5, 17 $|$ 22) (4, 19 $|$ 23) (8, 16 $|$ 24) \\ \hline
\multirow{4}{*}{4} & \multirow{2}{*}{9} & \multirow{2}{*}{22} & \multirow{2}{*}{19} & (12, 7, \underline{19}, \underline{20}), \{19, 26, 32, 39\} & (3, 4 $|$ 7) (2, 10 $|$ 12) (5, 15 $|$ 20) (1, 22 $|$ 23) (6, 18 $|$ 24) (9, 16 $|$ 25) (13, 14 $|$ 27) (11, 17 $|$ 28) (8, 21 $|$ 29) \\ \cline{2-6}
& \multirow{2}{*}{10} & \multirow{2}{*}{24} & \multirow{2}{*}{15} & (8, 24, \underline{15}, \underline{16}), \{24, 31, 32, 39\} & (2, 6 $|$ 8) (1, 14 $|$ 15) (3, 13 $|$ 16) (4, 21 $|$ 25) (7, 19 $|$ 26) (5, 22 $|$ 27) (11, 17 $|$ 28) (9, 20 $|$ 29) (12, 18 $|$ 30) (10, 23 $|$ 33)  \\ \hline
\multirow{9}{*}{5} & \multirow{3}{*}{11} & \multirow{3}{*}{27} & \multirow{3}{*}{26} & (3, 23, 11, \underline{26}, \underline{27}), \{26, 30, 34, 37, 53\} & (1, 2 $|$ 3) (5, 6 $|$ 11) (10, 13 $|$ 23) (9, 18 $|$ 27) (8, 20 $|$ 28) (4, 25 $|$ 29) (7, 24 $|$ 31) (15, 17 $|$ 32) (12, 21 $|$ 33) (16, 19 $|$ 35) (14, 22 $|$ 36) \\ \cline{2-6}
 & \multirow{3}{*}{12}  & \multirow{3}{*}{29} & \multirow{3}{*}{21} & (7, 29, 13, \underline{21}, \underline{23}), \{30, 34, 36, 42, 44\} & (2, 5 $|$ 7) (4, 9 $|$ 13) (1, 20 $|$ 21) (6, 17 $|$ 23) (3, 26 $|$ 29) (15, 16 $|$ 31) (8, 24 $|$ 32) (11, 22 $|$ 33) (10, 25 $|$ 35) (18, 19 $|$ 37) (12, 28 $|$ 40) (14, 27 $|$ 41) \\ \cline{2-6}
& \multirow{3}{*}{13} & \multirow{3}{*}{31} & \multirow{3}{*}{15} & (9, 5, 14, \underline{19}, \underline{23}), \{14, 19, 32, 33, 42\} & (1, 4 $|$ 5) (2, 7 $|$ 9) (3, 20 $|$ 23) (6, 28 $|$ 34) (8, 27 $|$ 35) (10, 26 $|$ 36) (12, 25 $|$ 37) (16, 22 $|$ 38) (18, 21 $|$ 39) (11, 29 $|$ 40) (17, 24 $|$ 41) (13, 30 $|$ 43) (15, 31 $|$ 46) \\ \hline
\multirow{6}{*}{6} & \multirow{3}{*}{14} & \multirow{3}{*}{34} & \multirow{3}{*}{28} & (12, 32, 8, 20, \underline{28}, \underline{30}), \{28, 40, 42, 44, 48, 58\} & (1, 7 $|$ 8) (3, 9 $|$ 12) (2, 18 $|$ 20) (4, 26 $|$ 30) (5, 27 $|$ 32) (6, 29 $|$ 35) (13, 23 $|$ 36) (15, 22 $|$ 37) (17, 21 $|$ 38) (14, 25 $|$ 39) (10, 31 $|$ 41) (19, 24 $|$ 43) (11, 34 $|$ 45) (16, 33 $|$ 49) \\ \cline{2-6}
 & \multirow{3}{*}{15} & \multirow{3}{*}{36} & \multirow{3}{*}{21} & (9, 29, 8, 13, \underline{21}, \underline{34}), \{21, 34, 37, 38, 43, 55\} & (1, 7 $|$ 8) (3, 6 $|$ 9) (2, 11 $|$ 13) (4, 25 $|$ 29) (12, 27 $|$ 39) (5, 35 $|$ 40) (10, 31 $|$ 41) (18, 24 $|$ 42) (14, 30 $|$ 44) (22, 23 $|$ 45) (20, 26 $|$ 46) (19, 28 $|$ 47) (15, 33 $|$ 48) (17, 32 $|$ 49) (16, 36 $|$ 52) \\ \hline
\multirow{11}{*}{7} & \multirow{3}{*}{16} & \multirow{3}{*}{39} & \multirow{3}{*}{36} & (6, 7, 13, 24, 20, \underline{36}, \underline{37}), $\{13, 20,$ $37, 43, 44, 56, 73\}$ & (1, 5 $|$ 6) (3, 4 $|$ 7) (2, 22 $|$ 24) (8, 28 $|$ 36) (15, 25 $|$ 40) (14, 27 $|$ 41) (9, 33 $|$ 42) (19, 26 $|$ 45) (12, 34 $|$ 46) (18, 29 $|$ 47) (10, 38 $|$ 48) (17, 32 $|$ 49) (11, 39 $|$ 50) (16, 35 $|$ 51) (21, 31 $|$ 52) (23, 30 $|$ 53)  \\ \cline{2-6}
& \multirow{4}{*}{17} & \multirow{4}{*}{41} & \multirow{4}{*}{28} & (6, 7, 13, 16, 20, \underline{29}, \underline{36}), \{13, 20, 29, 36, 42, 49, 65\} & (1, 5 $|$ 6) (3, 4 $|$ 7) (2, 14 $|$ 16) (8, 35 $|$ 43) (19, 25 $|$ 44) (11, 34 $|$ 45) (9, 37 $|$ 46) (17, 30 $|$ 47) (10, 38 $|$ 48) (22, 28 $|$ 50) (24, 27 $|$ 51) (21, 31 $|$ 52) (12, 41 $|$ 53) (15, 39 $|$ 54) (23, 32 $|$ 55) (18, 40 $|$ 58) (26, 33 $|$ 59)  \\ \cline{2-6}
& \multirow{4}{*}{18} & \multirow{4}{*}{43} & \multirow{4}{*}{19} & (13, 39, 6, 7, 12, \underline{19}, \underline{31}), \{13, 19, 31, 44, 45, 50, 52\} & (1, 5 $|$ 6) (3, 4 $|$ 7) (2, 10 $|$ 12) (9, 30 $|$ 39) (8, 38 $|$ 46) (20, 27 $|$ 47) (11, 37 $|$ 48) (15, 34 $|$ 49) (23, 28 $|$ 51) (18, 35 $|$ 53) (25, 29 $|$ 54) (22, 33 $|$ 55) (14, 42 $|$ 56) (17, 40 $|$ 57) (26, 32 $|$ 58) (16, 43 $|$ 59) (24, 36 $|$ 60) (21, 41 $|$ 62)  \\ \hline
\multirow{8}{*}{8} & \multirow{4}{*}{19} & \multirow{4}{*}{46} & \multirow{4}{*}{36} & (15, 16, 17, 19, 31, 33, \underline{36}, \underline{37}), \{31, 33, 36, 50, 52, 64, 69, 73\} & (1, 14 $|$ 15) (3, 13 $|$ 16) (5, 12 $|$ 17) (8, 11 $|$ 19) (2, 35 $|$ 37) (4, 43 $|$ 47) (10, 38 $|$ 48) (7, 42 $|$ 49) (6, 45 $|$ 51) (26, 27 $|$ 53) (25, 29 $|$ 54) (9, 46 $|$ 55) (24, 32 $|$ 56) (23, 34 $|$ 57) (28, 30 $|$ 58) (18, 41 $|$ 59) (20, 40 $|$ 60) (22, 39 $|$ 61) (21, 44 $|$ 65)  \\ \cline{2-6}
& \multirow{4}{*}{20} & \multirow{4}{*}{48} & \multirow{4}{*}{26} & (15, 16, 17, 18, 31, 23, \underline{33}, \underline{35}), \{31, 33, 35, 49, 50, 54, 56, 68\} & (2, 13 $|$ 15) (4, 12 $|$ 16) (3, 14 $|$ 17) (8, 10 $|$ 18) (1, 22 $|$ 23) (5, 46 $|$ 51) (7, 45 $|$ 52) (6, 47 $|$ 53) (11, 44 $|$ 55) (9, 48 $|$ 57) (20, 38 $|$ 58) (27, 32 $|$ 59) (26, 34 $|$ 60) (25, 36 $|$ 61) (19, 43 $|$ 62) (21, 42 $|$ 63) (24, 40 $|$ 64) (29, 37 $|$ 66) (28, 39 $|$ 67) (30, 41 $|$ 71)  \\ \hline
\multirow{9}{*}{9} & \multirow{4}{*}{21} & \multirow{4}{*}{51} & \multirow{4}{*}{45} & (18, 3, 21, 24, 22, 30, 23, \underline{45}, \underline{46}), \{21, 24, 45, 46, 52, 53, 64, 68, 91\} & (1, 2 $|$ 3) (6, 12 $|$ 18) (5, 17 $|$ 22) (7, 16 $|$ 23) (4, 26 $|$ 30) (10, 44 $|$ 54) (8, 47 $|$ 55) (14, 42 $|$ 56) (9, 48 $|$ 57) (15, 43 $|$ 58) (19, 40 $|$ 59) (11, 49 $|$ 60) (20, 41 $|$ 61) (27, 35 $|$ 62) (13, 50 $|$ 63) (28, 37 $|$ 65) (32, 34 $|$ 66) (29, 38 $|$ 67) (33, 36 $|$ 69) (31, 39 $|$ 70) (25, 51 $|$ 76) \\ \cline{2-6}
& \multirow{5}{*}{22} & \multirow{5}{*}{53} & \multirow{5}{*}{34} & (19, 16, 3, 13, 18, 31, 34, \underline{35}, \underline{49}), \{16, 19, 31, 35, 49, 65, 68, 69, 84\} & (1, 2 $|$ 3) (5, 8 $|$ 13) (6, 12 $|$ 18) (4, 30 $|$ 34) (7, 47 $|$ 54) (9, 46 $|$ 55) (11, 45 $|$ 56) (14, 43 $|$ 57) (10, 48 $|$ 58) (15, 44 $|$ 59) (20, 40 $|$ 60) (24, 37 $|$ 61) (29, 33 $|$ 62) (27, 36 $|$ 63) (26, 38 $|$ 64) (25, 41 $|$ 66) (17, 50 $|$ 67) (28, 42 $|$ 70) (32, 39 $|$ 71) (21, 51 $|$ 72) (22, 52 $|$ 74) (23, 53 $|$ 76) \\ \hline
\caption{\fontsize{12pt}{12pt}\selectfont Antimagic labelings for  \Href{labelings for cycles}.}
\label{am labeling}
\end{longtable}}
%\parker{I edited the labeling $f_t$ when $t=13,14,15$.}

\begin{table}[h!]
\centering
\fontsize{10pt}{9pt}\selectfont
\begin{tabular}{|c|c|c|p{3.5cm}|p{9.4cm}|}
\hline
\multirow{2}{*}{$t$} & \multirow{2}{*}{$m'$}& \multirow{2}{*}{$t+s-l$}&$g_t(E(2C_3))$, $\{\phi_{g_t}(V(2C_3))\}$ &\multicolumn{1}{c|}{\multirow{2}{*}{Pairs of labels on $P_3$ and their sums}} \\ \hline
\multirow{2}{*}{9} & \multirow{2}{*}{24} & \multirow{2}{*}{48} & (20, \underline{23}, \underline{24}), (12, 7, 19), \{19, 26, 31, 43, 44, 47\} &(3, 4 $|$ 7) (2, 10 $|$ 12) (5, 15 $|$ 20) (1, 22 $|$ 23) (6, 18 $|$ 24) (9, 16 $|$ 25) (13, 14 $|$ 27) (11, 17 $|$ 28) (8, 21 $|$ 29)\\\hline 
\multirow{2}{*}{10} & \multirow{2}{*}{26} & \multirow{2}{*}{46} & (8, \underline{25}, \underline{26}) (15, 16, 24), \{31, 33, 34, 39, 40, 51\} & (1, 7 $|$ 8), (2, 23 $|$ 25), (3, 12 $|$ 15), (4, 22 $|$ 26), (5, 11 $|$ 16), (6, 18 $|$ 24), (9, 19 $|$ 28), (10, 20 $|$ 30), (13, 14 $|$ 27), (17, 21 $|$ 38) \\ \hline
\multirow{3}{*}{11} & \multirow{3}{*}{28} & \multirow{3}{*}{43} &(3, \underline{27}, \underline{28}) (11, 23, 26), \{30, 31, 34, 37, 49, 55\} & (1, 2 $|$ 3), (4, 24 $|$ 28), (5, 6 $|$ 11), (7, 20 $|$ 27), (8, 15 $|$ 23), (9, 17 $|$ 26), (10, 19 $|$ 29), (12, 21 $|$ 33), (13, 22 $|$ 35), (14, 18 $|$ 32), (16, 25 $|$ 41) \\ \hline
\multirow{3}{*}{12} & \multirow{3}{*}{30} & \multirow{3}{*}{39} & (3, \underline{29}, \underline{30}) (11, 23, 26), \{32, 33, 34, 37, 49, 59\} & (1, 2 $|$ 3), (4, 25 $|$ 29), (5, 6 $|$ 11), (7, 19 $|$ 26), (8, 22 $|$ 30), (9, 14 $|$ 23), (10, 28 $|$ 38), (12, 24 $|$ 36), (13, 27 $|$ 40), (15, 16 $|$ 31), (17, 18 $|$ 35), (20, 21 $|$ 41) \\ \hline
% use \multirow{2}{*}{4} if you want this to float in the middle 
\multirow{3}{*}{13} & \multirow{3}{*}{32} & \multirow{3}{*}{34} &(3, \underline{31}, \underline{32}) (15, 16, 17), \{31, 32, 33, 34, 35, 63\} & (1, 2$|$ 3), (4, 11$|$ 15), (5, 12 $|$ 17), (6, 10 $|$ 16), (7, 29 $|$ 36), (8, 30 $|$ 38), (9, 28 $|$ 37), (13, 26 $|$ 39), (14, 27 $|$ 41), (18, 22 $|$ 40), (19, 23 $|$ 42), (20, 24 $|$ 44), (21, 25 $|$ 46) \\ \hline
\multirow{3}{*}{14} & \multirow{3}{*}{34} & \multirow{3}{*}{28} & (10, 18, 34), (6, \underline{28}, \underline{29}), \{28, 34, 35, 44, 52, 57\} &  (2, 4 $|$ 6), (1, 9 $|$ 10), (3, 15 $|$ 18), (7, 22 $|$ 29), (5, 31 $|$ 36), (16, 21 $|$ 37), (11, 27 $|$ 38), (13, 26 $|$ 39), (17, 23 $|$ 40), (8, 33 $|$ 41), (12, 30 $|$ 42), (19, 24 $|$ 43), (20, 25 $|$ 45), (14, 32 $|$ 46) \\ \hline
\multirow{4}{*}{15} &\multirow{4}{*}{36} & \multirow{4}{*}{21} &(5, 9, 35) (14, \underline{21}, \underline{36}), \{14, 35, 40, 44, 50, 57\} & (1, 4 $|$ 5), (2, 7 $|$ 9), (3, 18 $|$ 21), (6, 30 $|$ 36), (8, 29 $|$ 37), (10, 28 $|$ 38), (11, 31 $|$ 42), (12, 27 $|$ 39), (13, 33 $|$ 46), (15, 34 $|$ 49), (16, 32 $|$ 48), (17, 24 $|$ 41), (19, 26 $|$ 45), (20, 23 $|$ 43), (22, 25 $|$ 47)   \\ \hline
%\hline
%\multirow{4}{*}{14} & \multirow{4}{*}{(28, 29, 6) (10, 18, 34) \{57, 35, 34, 28, 52, 44\}} & (2, 4 $|$ 6), (1, 9 $|$ 10), (3, 15 $|$ 18), (7, 22 $|$ 29), (5, 31 $|$ 36), (16, 21 $|$ 37), (11, 27 $|$ 38), (13, 26 $|$ 39), (17, 23 $|$ 40), (8, 33 $|$ 41), (12, 30 $|$ 42), (19, 24 $|$ 43), (20, 25 $|$ 45), (14, 32 $|$ 46) \\ \hline
\end{tabular}
\caption{Antimagic labelings $g_t$ of $2C_3 \cup tP_3$ when $9 \leq t \leq 15$. The underlined numbers for $9 \leq t \leq 13$ are $m'$ and $m'-1$ (satisfying the hypotheses of \Href{two largest labels general}); while for $t=14, 15$, they are labels greater than or equal to $t+s-l$ (satisfying the hypotheses of \Href{lemma am general}).} %\parker{$t+s-l$}.}
\label{tab:doube c3 B}
\end{table}
\end{document}